\documentclass[%
  a4paper,
  onecolumn,
]{mypreprint}

\usepackage[english]{babel}

\usepackage{cleveref}

\usepackage{amsmath}
\usepackage{amssymb}
\usepackage{amsthm}

\usepackage[linesnumbered, lined, algoruled]{algorithm2e}
\Crefname{algocf}{Algorithm}{Algorithms}

\usepackage{graphicx}

\usepackage{subcaption}

\usepackage{tikz}
  \usetikzlibrary{backgrounds}
  \usetikzlibrary{calc}
  \usetikzlibrary{patterns}
  \usetikzlibrary{positioning}
  \usetikzlibrary{shapes}
  
\usepackage{pgfplots}
\usepackage{pgfplotstable}
  \pgfplotsset{%
    compat = 1.13,
    colormap name = viridis
  }
  \usepgfplotslibrary{external}
  \tikzset{external/system call = {%
    pdflatex \tikzexternalcheckshellescape
      -halt-on-error
      -interaction=batchmode
      -jobname "\image" "\texsource"}}
  \tikzexternaldisable

\theoremstyle{plain}\newtheorem{corollary}{Corollary}
\theoremstyle{plain}\newtheorem{lemma}{Lemma}
\theoremstyle{plain}\newtheorem{proposition}{Proposition}
\theoremstyle{plain}\newtheorem{theorem}{Theorem}

\renewcommand{\rm}[1]{\ensuremath{\mathrm{#1}}}
\renewcommand{\sf}[1]{\ensuremath{\mathsf{#1}}}

\DeclareMathOperator{\mspan}{span}
\newcommand{\trans}{\ensuremath{\mkern-1.5mu\sf{T}}}

\newcommand{\C}{\ensuremath{\mathbb{C}}}
\newcommand{\R}{\ensuremath{\mathbb{R}}}
\newcommand{\N}{\ensuremath{\mathbb{N}}}

\newcommand{\nh}{\ensuremath{N}}
\newcommand{\nmin}{\ensuremath{n_{\min}}}

\newcommand{\np}{\ensuremath{p}}
\newcommand{\dT}{\ensuremath{T}}
\newcommand{\neig}{\ensuremath{n_{\rm{u}}}}
\newcommand{\neigr}{\ensuremath{r}}

\newcommand{\Xcal}{\ensuremath{\mathcal{X}}}
\newcommand{\Xcaln}{\ensuremath{\mathcal{X}_{0}^{\rm{n}}}}

\newcommand{\Vcal}{\ensuremath{\mathcal{V}}}
\newcommand{\Wcal}{\ensuremath{\mathcal{W}}}
\newcommand{\Wcalt}{\ensuremath{\widetilde{\mathcal{W}}}}
\newcommand{\Fcal}{\ensuremath{\mathcal{F}}}

\newcommand{\Xn}{\ensuremath{X^{\rm{n}}}}
\newcommand{\xn}{\ensuremath{x^{\rm{n}}}}
\newcommand{\dxn}{\ensuremath{\frac{\rm{d}}{\rm{d} t} \xn}}
\newcommand{\Un}{\ensuremath{U^{\rm{n}}}}
\newcommand{\un}{\ensuremath{u^{\rm{n}}}}

\newcommand{\Ar}{\ensuremath{\widehat{A}}}
\newcommand{\Br}{\ensuremath{\widehat{B}}}

\newcommand{\Kr}{\ensuremath{\widehat{K}}}
\newcommand{\Xr}{\ensuremath{\widehat{X}}}

\newcommand{\Sigmar}{\ensuremath{\widehat{\Sigma}}}

\newcommand{\Uc}{\ensuremath{\check{U}}}
\newcommand{\Vc}{\ensuremath{\check{V}}}

\newcommand{\Sigmac}{\ensuremath{\check{\Sigma}}}

\newcommand{\Wt}{\ensuremath{\widetilde{W}}}

\newcommand{\Xt}{\ensuremath{\widetilde{X}}}

\newcommand{\dx}{\ensuremath{\frac{\rm{d}}{\rm{d} t} x}}
\newcommand{\xb}{\ensuremath{\bar{x}}}
\newcommand{\ub}{\ensuremath{\bar{u}}}
\newcommand{\Ug}{\ensuremath{\breve{U}}}
\newcommand{\Xg}{\ensuremath{\breve{X}}}

\newcommand{\datatrip}{(U_{-}, X_{-}, X_{+})}
\newcommand{\datatripn}{(\Un_{-}, \Xn_{-}, \Xn_{+})}
\newcommand{\datatripg}{(\Ug_{-}, \Xg_{-}, \Xg_{+})}

\newcommand{\datatripr}{(U_{-}, \Xr_{-}, \Xr_{+})}

\definecolor{matlabblue}{HTML}{0072BD}
\definecolor{matlaborange}{HTML}{D95319}
\definecolor{matlabyellow}{HTML}{EDB120}
\definecolor{matlabpurple}{HTML}{7E2F8E}
\definecolor{matlabgreen}{HTML}{77AC30}
\definecolor{matlablightblue}{HTML}{4DBEEE}
\definecolor{matlabred}{HTML}{A2142F}

\definecolor{cblred}{HTML}{E41A1C}
\definecolor{cblpurple}{HTML}{984EA3}
\definecolor{mydarkgreen}{HTML}{00624C}

\tikzstyle{sline} = [
  solid,
  line width = 1.5pt
]

\tikzstyle{rline} = [
  solid,
  line width = 1.7pt
]

\tikzstyle{dline} = [
  dashed,
  line width = 1.5pt
]

\tikzstyle{samplecol} = [
  line width   = 1.3pt,
  fill opacity = 0.75
]

\newcommand{\plotfontsize}{\small}

\pgfplotscreateplotcyclelist{stateslist}{
  {matlabblue, sline},
  {matlaborange, sline},
  {matlabpurple, sline},
  {matlabgreen, sline}
}

\pgfplotscreateplotcyclelist{samplelist}{
  {draw = matlabblue, fill = matlabblue, samplecol},
  {draw = matlaborange, fill = matlaborange, samplecol},
  {draw = matlabpurple, fill = matlabpurple, samplecol},
  {draw = matlabred, fill = matlabred, samplecol},
  {draw = black, fill = black, samplecol}
}

\pgfplotscreateplotcyclelist{trainlist}{
  {cblred, rline},
  {mydarkgreen, rline},
  {cblpurple, rline}
}

\begin{document}


\title{Context-aware controller inference for stabilizing dynamical systems
  from scarce data}
  
\author[$\ast$,1]{Steffen W. R. Werner}
\affil[$\ast$]{Courant Institute of Mathematical Sciences, New York
  University, New York, NY 10012, USA}
\affil[1]{\email{steffen.werner@nyu.edu}, \orcid{0000-0003-1667-4862}}
  
\author[$\ast$,2]{Benjamin Peherstorfer}
\affil[2]{\email{pehersto@cims.nyu.edu}, \orcid{0000-0002-1558-6775}}

\shorttitle{Context-aware controller inference for stabilization}
\shortauthor{S. W. R. Werner, B. Peherstorfer}
\shortdate{2023-01-18}
\shortinstitute{}
  


\abstract{%
  This work introduces a data-driven control approach for stabilizing
  high-dimensional dynamical systems from scarce data.
  The proposed context-aware controller inference approach is based on the
  observation that controllers need to act locally only on the unstable dynamics
  to stabilize systems.
  This means it is sufficient to learn the unstable dynamics alone, which are
  typically confined to much lower dimensional spaces than the high-dimensional
  state spaces of all system dynamics and thus few data samples are
  sufficient to identify them.
  Numerical experiments demonstrate that context-aware controller inference
  learns stabilizing controllers from orders of magnitude fewer data samples
  than traditional data-driven control techniques and variants of
  reinforcement learning.
  The experiments further show that the low data requirements of context-aware
  controller inference are especially beneficial in data-scarce engineering
  problems with complex physics, for which learning complete system dynamics is
  often intractable in terms of data and training costs.
}

\novelty{}

\maketitle


\section{Introduction}%
\label{sec:intro}

The design of feedback controllers for stabilizing dynamical systems is a
ubiquitous task in science and engineering.
Standard control techniques rely on the availability of models of the system
dynamics to construct controllers.
If models are unavailable, then typically models of system dynamics are learned
from data first and then standard control techniques are applied to the learned
models~\cite{BruK19}.
However, learning models of complex system dynamics can require large
amounts of data because learning models means identifying generic descriptions
that often also include information about the systems that are unnecessary for
the specific task of finding stabilizing controllers.
Additionally, collecting data from unstable systems is challenging because
without a stabilizing controller the system cannot be observed for a long time
before the dynamics become unstable and thus data collection becomes
uninformative.
In this work, we propose context-aware controller inference that stabilizes
systems based on the unstable dynamics that are learned from few state
observations by leveraging that unstable dynamics typically evolve in spaces of
much lower dimension than the dimension of state spaces of all---stable
and unstable---dynamics.
We show that bases of the spaces of unstable dynamics can be efficiently
estimated from derivative information of systems.
The corresponding numerical procedure of context-aware controller inference
achieves orders of magnitude reductions in the number of data samples that
are required for finding stabilizing controllers compared to traditional
data-driven methods that first identify models of the full dynamics that
are subsequently stabilized.

Many approaches for model-free, data-driven controller design based on
parameter tuning have been developed; see,
e.g.,~\cite{CamLS02, FliJ13, LeqGMetal03, SafT95}.
However, these methods are limited to systems with a small
number of observables and inputs.
In machine learning, reinforcement
learning~\cite{SilLHetal14, LilHPetal15, PerUS21} has been successfully applied
for the data-driven design of controllers, employing similar ideas as in the
parameter tuning methods mentioned above.
With the development of model reduction, efficient approaches for
modeling reduced dynamical systems from data have been developed, such as
dynamic mode decomposition and operator
inference~\cite{Sch10, TuRLetal14, PehW16, Peh20, SwiKHetal20}, sparse
identification methods~\cite{BruPK16, SchTW18}, and
the Loewner framework~\cite{MayA07,SchU16,PehGW17,SchUBetal18, AntGI16,GosA18}.
These methods inform many data-driven controller techniques that first identify
a model of the dynamics that is then stabilized by classical control approaches.
However, it has been shown that less data are required for the task of
stabilization than for the identification of models, which is the motivation for
this work~\cite{DePT20, VanETetal20, WerP22, VanD96}.

We introduce context-aware controller inference, which is a new data-driven
approach for learning stabilizing controllers from scarce data and that is
applicable to systems with nonlinear dynamics.
Context-aware controller inference exploits that controllers need to act only on
the unstable parts of system dynamics for
stabilization~\cite{BenCQ01, BenCQetal00, HeM94, Sim96} and that spaces in which
the unstable dynamics evolve can be estimated efficiently and cheaply from
gradients that are obtained from adjoints.
If $\neigr$ is the dimension of the space induced by the unstable dynamics, then
there always exist $\neigr$ states that are sufficient to be observed for
inferring a feedback controller that is guaranteed to stabilize the system.
The dimension $\neigr$ of the space induced by the unstable dynamics is
typically orders of magnitude lower than the dimension of the whole state space
in which all dynamics evolve and which often determines the high data and
training costs of traditional data-driven control methods.
Context-aware controller inference thus opens the door to stabilizing
systems from scarce data, such as near rare events and when data generation is
expensive, even if systems describe complex physics.


\section{Preliminaries}%
\label{sec:nonlinear}


\subsection{Stabilizing dynamical processes with feedback control}%
\label{subsec:stabfeedback}

Consider a system that gives rise to a process $(\xn(t))_{t \geq 0}$ that is
controlled by inputs $(\un(t))_{t \geq 0}$, where $\xn(t) \in \R^{\nh}$ and
$\un(t) \in \R^{\np}$ are the state and input at time $t$, respectively.
The feasible initial conditions $\xn(0) \in \Xcaln$ are in the subspace
$\Xcaln \subset \R^{\nh}$.
If the process is continuous in time, then the dynamics are governed by ordinary
differential equations
\begin{equation} \label{eqn:ctnsys}
  \begin{aligned}
    \dxn(t) & = f\big(\xn(t), \un(t)\big), & t \geq 0,
  \end{aligned}
\end{equation}
with the potentially nonlinear right-hand side function
$f\colon \R^{\nh} \times \R^{\np} \rightarrow \R^{\nh}$.  
In discrete time, the process is governed by the difference equations 
\begin{equation} \label{eqn:dtnsys}
  \begin{aligned}
    \xn(t + 1) & = f\big(\xn(t), \un(t)\big), & t \in \N_{0} := \N \cup \{0\}.
  \end{aligned}
\end{equation}
Let $(\xb, \ub) \in \R^{\nh} \times \R^{\np}$ be a controlled, constant-in-time
steady state---equilibrium point---such that, in the continuous-time case,
\begin{equation} \label{eqn:steadystatect}
  f(\xb, \ub) = 0,
\end{equation}
and in the discrete-time case
\begin{equation} \label{eqn:steadystatedt}
  f(\xb, \ub) = \xb
\end{equation}
hold.
Conditions~\cref{eqn:steadystatect,eqn:steadystatedt}
imply that the steady state is constant over time.
In the following, we assume that the nonlinear function $f$
in~\cref{eqn:ctnsys,eqn:dtnsys} is analytic at $(\xb, \ub)$.

A steady state $(\xb, \ub)$ is unstable if, for the fixed input signal $\ub$,
trajectories diverge away from $\xb$ for initial conditions $\xn(0)$ in a
neighborhood of the steady state $\xb$.
A linear state-feedback controller $K \in \R^{\np \times \nh}$ asymptotically
stabilizes trajectories in sufficiently small neighborhoods about $\xb$ if, for
a given $\xn(0) \in \Xcal_{0}$ in a neighborhood of $\xb$, the input
trajectory $(\un(t))_{t \geq 0}$ with
$\un(t) = K(\xn(t) - \xb) + \ub$ leads to $\xn(t) \to \xb$
for $t \to \infty$, where $(\xn(t))_{t \geq 0}$ is given by
either~\cref{eqn:ctnsys} or~\cref{eqn:dtnsys}; see~\cite{NijV16}.


\subsection{Problem formulation}%
\label{subsec:problem}

In this work, a model of the system in form of the right-hand side function $f$
in~\cref{eqn:ctnsys} (or~\cref{eqn:dtnsys}) is unavailable, which means the
function $f$ cannot be evaluated directly and is not given in closed form.
Instead, we either can generate data triplets $\datatripn$ by querying
the system at feasible initial conditions and inputs in a black-box fashion or
have given $\datatripn$.
In the continuous-time case, the data triplets are
\begin{align*}
  \Un_{-} & = \begin{bmatrix} \un(t_{0})  & \ldots &
    \un(t_{\dT - 1}) \end{bmatrix} \in \R^{\np \times \dT}, \quad  
  \Xn_{-} = \begin{bmatrix} \xn(t_{0}) &  \ldots &
    \xn(t_{\dT - 1}) \end{bmatrix} \in \R^{\nh \times \dT}, \\
  \Xn_{+} & = \begin{bmatrix} \dxn(t_{0}) &  \ldots &
    \dxn(t_{\dT - 1}) \end{bmatrix} \in \R^{\nh \times \dT},
\end{align*}
at discrete time points $0 \leq t_{0} < t_{1} < \ldots < t_{\dT - 1}$.
In the discrete-time case, the data triplets are 
\begin{align*}
  \Un_{-} & = \begin{bmatrix} \un(0)  & \ldots & \un(\dT - 1)
    \end{bmatrix} \in \R^{\np \times \dT},\quad
  \Xn_{-} = \begin{bmatrix} \xn(0)  & \ldots & \xn(T - 1)
    \end{bmatrix} \in \R^{\nh \times \dT}, \\
  \Xn_{+} & = \begin{bmatrix} \xn(1) & \ldots & \xn(T)
    \end{bmatrix} \in \R^{\nh \times \dT}.
\end{align*}
A data triplet $\datatripn$ can also contain the concatenation of several
trajectories, which is not reflected in the notation for ease of exposition.
We seek to construct controllers $K$ from data triplets
$\datatripn$ to stabilize systems about steady states of interest $(\xb, \ub)$.

The traditional approach to data-driven control is first learning a model of the
system dynamics and then applying standard techniques from control.
However, this traditional two-step learn-then-stabilize approach can be
expensive in terms of the number of required state observations, which typically
scales with the dimension of the
system~\cite{DePT20, VanETetal20, WerP22,VanD96}.
In particular, if observing states is expensive and the underlying dynamics
describe complex physical phenomena, then often infeasibly high numbers of state
observations are required to learn models of the system dynamics, which makes
traditional learn-then-stabilize approaches intractable.


\section{Inferring controllers from data of unstable dynamics}%
\label{sec:inferpartstab}

We introduce context-aware controller inference that guarantees learning
stabilizing controllers with data requirements that scale as the
dimension of the subspace spanned by the unstable dynamics, which is
typically orders of magnitude lower than the dimension of the whole state space
in which all dynamics evolve.
To achieve this, we exploit that controllers need to act only on the unstable
part of the system dynamics for stabilization.


\subsection{Stabilizing nonlinear systems near steady states}%
\label{subsec:stabnl}

Let $(\xb, \ub)$ be a steady state.
The right-hand side function $f$ of~\cref{eqn:ctnsys} (or~\cref{eqn:dtnsys}) can
be approximated at $(\xn(t), \un(t))$ with the first two terms of its
Taylor series expansion about $(\xb, \ub)$:
\begin{equation} \label{eqn:taylor}
  \begin{aligned}
    f\big(\xn(t), \un(t)\big) = f(\xb, \ub) +
      \partial_{\rm{x}} f(\xb, \ub)
      \big(\xn(t) - \xb\big) + 
      \partial_{\rm{u}} f(\xb, \ub) \big(\un(t) - \ub\big)\\
    {}+{}
      \mathcal{O}\left(\big(\xn(t) - \xb\big)^{2}\right) +
      \mathcal{O}\left(\big(\un(t) - \ub\big)^{2}\right),
  \end{aligned}
\end{equation}
where $\partial_{\rm{x}} f(\xb, \ub)$ and $\partial_{\rm{u}} f(\xb, \ub)$
denote the parts of the Jacobian at $(\xb, \ub)$ of $f$ with respect to $\xn(t)$
and $\un(t)$, respectively.
The first term $f(\xb, \ub)$ in~\cref{eqn:taylor} is either $f(\xb, \ub) = 0$
or $f(\xb, \ub) = \xb$; cf.~\Cref{eqn:steadystatect,eqn:steadystatedt}.
With the system matrices $A = \partial_{\rm{x}} f(\xb, \ub) \in
\R^{\nh \times \nh}$ and $B = \partial_{\rm{u}} f(\xb, \ub) \in
\R^{\nh \times \np}$, a model of the linearized system
of~\cref{eqn:ctnsys} about the steady state is 
\begin{equation} \label{eqn:ctsys}
  \begin{aligned}
    \dx(t) & = A x(t) + B u(t), & t \geq 0,
  \end{aligned}
\end{equation}
and analogously for the discrete-time case~\cref{eqn:dtnsys},
\begin{equation} \label{eqn:dtsys}
  \begin{aligned}
    x(t + 1) & = A x(t) + B u(t), & t \in \N_{0},
  \end{aligned}
\end{equation}
where $x(t)$ approximates the shifted nonlinear state $\xn(t) - \bar{x}$ at time
$t$ and the input $u(t)$ corresponds to $\un(t) - \bar{u}$.

The system described by the model~\cref{eqn:ctsys} (or~\cref{eqn:dtsys}) is
called asymptotically stabilizable if there exists a constant matrix
$K \in \R^{\np \times \nh}$ such that the application of the state feedback
$u(t) = K x(t)$ stabilizes the system, i.e.,
$\lVert x(t) \rVert \rightarrow 0$ for $t \rightarrow \infty$.
A linear model such as~\cref{eqn:ctsys} is asymptotically stable if all
eigenvalues of $A$ have negative real parts; in case of discrete-time linear
models~\cref{eqn:dtsys}, the eigenvalues of $A$ have to be inside the unit disc.
A linear state-space model is called stabilizable if there exists a $K$ such
that $A + B K$ has only stable eigenvalues.

The following proposition states that a controller $K$ obtained from
the linearized system is guaranteed to stabilize the nonlinear system if the
state $\xn(t)$ stays in the neighborhood of the steady state $\xb$.

\begin{proposition}[Locally stabilizing feedback controller;
  {cf.~\cite[Sec.~10.1]{NijV16}}]%
  \label{prp:stabnl}
  Consider a linearized system described
  by~\cref{eqn:ctsys} (or~\cref{eqn:dtsys}) 
  of a corresponding nonlinear system described by~\cref{eqn:ctnsys}
  (or~\cref{eqn:dtnsys}) at the steady state of interest $(\xb, \ub)$, where the
  nonlinear function $f$ is analytic.
  If $K$ is a controller that stabilizes the linearized system asymptotically,
  then $\un(t)  = K (\xn(t) - \xb) + \ub$ locally stabilizes the nonlinear
  system about $(\xb, \ub)$ in the asymptotic sense, i.e.,
  there exists an $\epsilon > 0$ such that for all $\xn(0) \in \Xcaln(0)$ with 
  $\lVert \xb - \xn(0) \rVert < \epsilon$ we have that
  $\xn(t) \rightarrow \xb$ for $t \rightarrow \infty$.
\end{proposition}


\subsection{Stabilizing models based on unstable eigenvalues}%
\label{subsec:partstab}

To restrict the action of a controller $K$ to the unstable
eigenvalues of the spectrum of linear models, we build on the theory of
pole placement~\cite{Var81, Saa88} and partial
stabilization~\cite{BenCQ01, BenCQetal00, HeM94, Sim96}.
This leads to a reduction of the dimension of the spaces in which the dynamics
evolve, from a typically high-dimensional state space to subspaces spanned by
the unstable dynamics that are often orders of magnitude lower in dimension.
In this section, we generalize concepts of partial stabilization to be
applicable to the data-driven setting.

In the following, when we state that two matrices have the same spectrum, it
implies that they have the same eigenvalues with the same geometric and
algebraic multiplicities.

\begin{lemma}%
  \label{thm:partstableft}
  Consider the linear model $(A, B)$ given in~\cref{eqn:ctsys}
  (or~\cref{eqn:dtsys}) and assume it is stabilizable.
  Let $W \in \R^{\nh \times \neig}$ be a basis matrix of the
  $\neig$-dimensional left eigenspace of $A$ corresponding to the unstable
  eigenvalues and define the matrices
  \begin{equation} \label{eqn:partstableftTmp0}
    \begin{aligned}
      \Ar & =  W^{\trans} A \left( W^{\dagger} \right)^{\trans}, &
      \Br & =  W^{\trans} B,
    \end{aligned}
  \end{equation}
  where $W^{\dagger}$ is a left inverse of $W$.
  Let further $\Kr$ stabilize $(\Ar, \Br)$.
  Then, it holds that
  \begin{equation*}
    K  = \Kr W^{\trans}
  \end{equation*}
  stabilizes the model $(A, B)$. 
  Furthermore, it holds that the spectrum of $\Ar$ is the spectrum of $A$
  corresponding to the unstable eigenvalues and that the union of the stable
  spectrum of $A$ and the spectrum of $\Ar + \Br \Kr$ is the spectrum of
  $A + B K$.
\end{lemma}

\Cref{thm:partstableft} implies that the unstable dynamics of the model $(A, B)$
can be fully described by the space spanned by the columns of the basis
matrix $W$, which contains as columns the left eigenvectors of the unstable
eigenvalues of $A$.
The eigenvalues of $A$ can be complex such that also the corresponding
eigenvectors are complex-valued.
However, with the assumption that $A$ is real-valued, the eigenvalues and
eigenvectors appear in complex conjugate pairs such that there exists a
real basis matrix $W$ of the corresponding complex eigenspace;
see, e.g.,~\cite[Sec.~7.4.1]{GolV13}.
The last statement of \Cref{thm:partstableft} implies that only the unstable
eigenvalues of $A$ are changed in the closed-loop model by the application
of $K$.
In particular, the eigenvalues are determined via the spectrum of the
reduced closed-loop model $\Ar + \Br \Kr$ by the construction of $\Kr$.

\begin{proof}[Proof of \Cref{thm:partstableft}.]
  A basis matrix $\Wt \in \C^{\nh \times \neig}$ of the left eigenspace of $A$
  corresponding to the unstable eigenvalues is defined by
  \begin{equation} \label{eqn:partstableftTmp1}
    A^{\trans} \Wt = \Wt \Lambda^{\trans},
  \end{equation}
  where $\Lambda \in \C^{\neig \times \neig}$ consists of Jordan blocks with
  the unstable eigenvalues of $A$ on its diagonal.
  There exists a transformation $S \in \C^{\neig \times \neig}$ with
  $\Lambda = S^{-1} \Ar S$, where $\Ar$ is
  from~\cref{eqn:partstableftTmp0}, such that $\Ar$ has by construction the same
  spectrum as $\Lambda$.
  Inserting the transformation into~\cref{eqn:partstableftTmp1} yields
  \begin{equation*}
    \begin{aligned}
      A^{\trans} \Wt & = \Wt S^{\trans} \Ar^{\trans} S^{-\trans} &
      & \Leftrightarrow &
      A^{\trans} \Wt S^{\trans} & = \Wt S^{\trans} \Ar^{\trans}.
    \end{aligned}
  \end{equation*}
  Setting $W = \Wt S^{\trans}$ and transposing the eigenvalue relation
  we see that
  \begin{equation} \label{eqn:partstableftTmp4}
    \Ar = W^{\trans} A \left( W^{\dagger} \right)^{\trans}
  \end{equation}
  holds for any left inverse $W^{\dagger}$ of $W$.
  Let $W_{\perp}$ be a basis matrix of the orthogonal complement to the
  subspace spanned by the columns of $W$, i.e., it holds that
  \begin{equation*}
    \mspan\left(\begin{bmatrix} W & W_{\perp} \end{bmatrix}\right)
    = \R^{\nh},
  \end{equation*}
  with $W_{\perp}^{\trans} W = 0$ and $W^{\trans} W_{\perp} = 0$.
  Then, it holds that
  \begin{equation} \label{eqn:partstableftTmp2}
    \begin{aligned}
      \begin{bmatrix} W & W_{\perp} \end{bmatrix}^{\trans} A 
        \begin{bmatrix} W & W_{\perp} \end{bmatrix}^{-\trans} & =
      \begin{bmatrix} \Ar & 0 \\ A_{21} & A_{22} \end{bmatrix}, &
        \begin{bmatrix} W & W_{\perp} \end{bmatrix}^{\trans} B & = 
        \begin{bmatrix} \Br \\ B_{2} \end{bmatrix},
    \end{aligned}
  \end{equation}
  where $\Br$ is as in~\cref{eqn:partstableftTmp0} and the
  spectrum of $A_{22}$ is the stable spectrum of $A$;
  see, e.g.,~\cite[Thm.~7.1.6]{GolV13}.
  We can now construct a controller $\Kr$ such that $\Ar + \Br \Kr$ is
  stable since  $(A, B)$ has been assumed to be
  stabilizable.
  Augmenting the controller $\Kr$ by zeros yields the closed-loop
  matrix of~\cref{eqn:partstableftTmp2} as
  \begin{equation} \label{eqn:partstableftTmp3}
    \begin{bmatrix} \Ar & 0 \\ A_{21} & A_{22} \end{bmatrix} + 
      \begin{bmatrix} \Br \\ B_{2} \end{bmatrix}
      \begin{bmatrix} \Kr & 0\end{bmatrix} = 
      \begin{bmatrix} \Ar + \Br \Kr & 0 \\ A_{21} + B_{2} \Kr & A_{22}
      \end{bmatrix}.
  \end{equation}
  The closed-loop matrix in~\cref{eqn:partstableftTmp3} is stable since
  $\Kr$ has been chosen such that $\Ar + \Br \Kr$ is stable and $A_{22}$ has
  only stable eigenvalues.
  Inverting the transformation in~\cref{eqn:partstableftTmp2} and applying that
  to~\cref{eqn:partstableftTmp3} gives the closed-loop matrix of the original
  state-space model with
  \begin{equation*}
    A + B K = \begin{bmatrix} W & W_{\perp} \end{bmatrix}^{-\trans}
      \left( \begin{bmatrix} \Ar & 0 \\ A_{21} & A_{22} \end{bmatrix}
      + \begin{bmatrix} \Br \\ B_{2} \end{bmatrix}
      \begin{bmatrix} \Kr & 0\end{bmatrix} \right)
      \begin{bmatrix} W & W_{\perp} \end{bmatrix}^{\trans},
  \end{equation*}
  which reveals by multiplying out the block matrices that
  \begin{equation*}
    K = \begin{bmatrix} \Kr & 0\end{bmatrix}
      \begin{bmatrix} W & W_{\perp} \end{bmatrix}^{\trans}
      = \Kr W^{\trans}
  \end{equation*}
  holds.
  Since transformations do not change the spectrum of matrices, the spectrum of
  $A + B K$ is the spectrum of~\cref{eqn:partstableftTmp3}, which concludes
  the proof.
\end{proof}

In \Cref{thm:partstableft}, the left eigenvector basis is essential
for the proof.
It leads to the lower triangular form~\cref{eqn:partstableftTmp2} that has the
required ordering of the diagonal blocks with the unstable eigenvalues contained
in the upper and the stable eigenvalues in the lower diagonal block.
This particular ordering is necessary since otherwise the spectrum is disturbed
by the action of the controller; cf.~\cite[Sec.~3.2.2]{WerP22}.
However, the left eigenspace of the unstable eigenvalues can also be constructed
as orthogonal complement of the right eigenspace of the stable
eigenvalues, as the next lemma shows.

\begin{lemma}%
  \label{lmm:orthcomp}
  Given a matrix $A$, let $W \in \R^{\nh \times \neig}$ be a basis matrix of the
  left eigenspace corresponding to the unstable eigenvalues of $A$,
  and let $Q_{\perp} \in \R^{\nh \times \neig}$ be a basis matrix of the
  orthogonal complement to the right eigenspace of $A$ corresponding to the
  stable eigenvalues.
  Then, it holds that
  \begin{equation} \label{eqn:orthcompsub}
    \mspan(W) = \mspan\left( \left( Q_{\perp}^{\dagger}
      \right)^{\trans} \right),
  \end{equation}
  for all left inverses $Q_{\perp}^{\dagger}$ of $Q_{\perp}$.
  In the case $Q_{\perp}$ is an orthogonal basis, \Cref{eqn:orthcompsub}
  simplifies to
  \begin{equation*}
    \mspan(W) = \mspan(Q_{\perp}).
  \end{equation*}
\end{lemma}
\begin{proof}
  Let $Q$ be a basis matrix of the eigenspace of $A$ associated with the stable
  eigenvalues.
  By basis concatenation with $Q_{\perp}$, it holds
  \begin{equation} \label{eqn:orthcompTmp1}
    \begin{bmatrix} Q_{\perp} & Q \end{bmatrix}^{-1} A
      \begin{bmatrix} Q_{\perp} & Q \end{bmatrix} =
      \begin{bmatrix} \Ar & 0 \\ A_{21} & A_{22} \end{bmatrix},
  \end{equation}
  where the spectrum of $\Ar$ is the unstable spectrum of $A$ and the
  spectrum of $A_{22}$ is the stable one of $A$;
  see~\cite[Thm.~7.1.6]{GolV13}.
  In particular, we get from~\cref{eqn:orthcompTmp1} that
  \begin{equation} \label{eqn:orthcompTmp2}
    Q_{\perp}^{\dagger} A Q_{\perp} = \Ar,
  \end{equation}
  for all left inverses $Q_{\perp}^{\dagger}$ of $Q_{\perp}$.
  The basis matrix $W$ can be chosen such that the matrices $\Ar$ and $A_{22}$
  in~\cref{eqn:partstableftTmp2,eqn:orthcompTmp1} are identical, since both
  spectra are identical there must be an appropriate similarity transformation.
  Comparing the two basis matrices
  in~\cref{eqn:orthcompTmp2,eqn:partstableftTmp4} yields the result.
\end{proof}

Note that there is a line of work on constructing manifolds of unstable
dynamics, instead of subspaces as in our approach. Working with manifolds, instead of spaces, can potentially be beneficial when
systems are strongly nonlinear.
The corresponding computational methods
involving manifolds can become computationally challenging, especially for
systems in high-dimensional state spaces~\cite{KraODetal05, ZieDG19}.
It remains future work to efficiently combine our control approach with methods
for constructing unstable manifolds.


\subsection{Controller inference from idealized data}%
\label{subsec:partstabdata}

In this section, we will construct stabilizing controllers from given
data triplets.
For now, we consider data triplets $\datatrip$ that are idealized in the sense
that they are obtained from models of the linearized systems.
However, in the next section and in the numerical experiments we will consider
data from the original nonlinear systems.
In the continuous-time case, the idealized data triplets are
\begin{equation*}
  \begin{aligned}
    U_{-} & = \begin{bmatrix} u(t_{0}) & \dots & u(t_{\dT - 1})
      \end{bmatrix}, &
    X_{-} & = \begin{bmatrix} x(t_{0}) & \dots &  x(t_{\dT - 1})
      \end{bmatrix}, \\
    X_{+} & = \begin{bmatrix} \dx(t_{0}) & \dots & \dx(t_{\dT - 1})
      \end{bmatrix},
  \end{aligned}
\end{equation*}
at discrete time points $0 \leq t_{0} < t_{1} < \ldots < t_{\dT - 1}$ with
states from~\cref{eqn:ctsys}, and in the discrete-time case, the idealized data
triplets with states from~\cref{eqn:dtsys} are
\begin{equation*}
  \begin{aligned}
    U_{-}  & = \begin{bmatrix} u(0) & \dots & u(T - 1) \end{bmatrix}, &
    X_{-}  & = \begin{bmatrix} x(0) & \dots & x(T-1) \end{bmatrix},\\
    X_{+}  & = \begin{bmatrix} x(1) & \dots & x(T) \end{bmatrix}.
  \end{aligned}
\end{equation*}

The fundamental principle of data informativity~\cite{VanETetal20} is that
potentially many linear models can describe a given data triplet $\datatrip$
and that a stabilizing controller needs to stabilize all those linear models to
guarantee the stabilization of the model from which the data has been
generated.
The set of all linear models that explain $\datatrip$ is
\begin{equation*}
  \Sigma_{\rm{i/s}} = \left\{ (A, B) \left|~ X_{+} = A X_{-} + B U_{-}
    \right.\right\},
\end{equation*}
and all linear models that are stabilized by a given $K$ are
\begin{equation*}
  \Sigma_{K} = \left\{ (A, B) \left|~ A + B K~\text{is stable}
    \right. \right\}.
\end{equation*}
A given data triplet $\datatrip$ is called \emph{informative for
stabilization} if there exists a controller $K$ that stabilizes all explaining
linear models, which means in terms of the two sets
$\Sigma_{\rm{i / s}}$ and $\Sigma_{\rm{K}}$ that $\Sigma_{\rm{i / s}} \subseteq
\Sigma_{\rm{K}}$ holds; see~\cite{VanETetal20}.

The matrix $A$ of the linear model $(A, B)$ can have components
corresponding to unstable eigenvalues that do not contribute to the system
dynamics, which means these are not excited by initial conditions from $\Xcal_0$
or the system inputs.
For any linear model $(A, B)$ with a basis matrix $X_{0}$ of the initial
conditions subspace $\Xcal_{0}$, there exists an invertible matrix
$S \in \R^{\nh \times \nh}$ to transform the system into the Kalman
controllability form:
\begin{equation} \label{eqn:kalman}
  \begin{aligned}
    S^{-1} A S & = \begin{bmatrix} A_{11} & A_{12} & A_{13}\\ 0 & A_{22} &
      A_{23}\\ 0 & 0 & A_{33} \end{bmatrix}, &
    S^{-1} B & = \begin{bmatrix} B_{1} \\ 0 \\ 0 \end{bmatrix},& 
    S^{-1} X_{0} & = \begin{bmatrix} X_{10} \\ X_{20} \\ 0 \end{bmatrix},
  \end{aligned}
\end{equation}
where the zero blocks in the transformed $B$ and $X_{0}$ are chosen as large as
possible; see~\cite{WerP22}.
The first block row in~\cref{eqn:kalman} corresponds to the controllable
dynamics, the second one to those only excited by the initial conditions and the
last block row are zero dynamics, which are neither excited by inputs nor
initial conditions.
Consequently, the eigenspaces corresponding to the unstable eigenvalues in
$A_{33}$ do not result in unstable dynamics since these components are never
excited.
In the following, we consider the left eigenbasis matrix
$W \in \R^{\nh \times \neigr}$ to be minimal in the sense that the subspace
$\Wcal$ spanned by $W$ contains the left eigenspaces corresponding to the
unstable eigenvalues in $A_{11}$ and $A_{22}$ in~\cref{eqn:kalman}.
If $A_{33}$ has unstable eigenvalues, the corresponding eigenbasis is not taken
into account.
In other words, we consider only those eigenspaces that describe unstable
dynamics that can be excited by initial conditions or inputs.

The following theorem characterizes the stabilization of models based on
unstable dynamics using the data informativity concept.

\begin{theorem}%
  \label{thm:partstabdata}
  Consider a potentially nonlinear system given by~\cref{eqn:ctnsys}
  (or~\cref{eqn:dtnsys}).
  Given $\datatrip$ a data triplet sampled from a linearized model obtained
  from the considered nonlinear system about the steady state $(\xb, \ub)$, with
  state-space dimension $\nh$.
  Let further $W \in \R^{\nh \times \neigr}$ be a basis matrix of the minimal
  left eigenspace of the unstable eigenvalues that correspond to the unstable
  system dynamics.
  If the data triplet $\datatripr$ is informative for stabilization, which
  means that $\Sigmar_{\rm{i/s}} \subseteq \Sigmar_{\Kr}$ holds for a suitable
  $\Kr \in \R^{\np \times \neigr}$ and
  \begin{equation} \label{eqn:datatrunc}
    \begin{aligned}
      \Xr_{-} & = W^{\trans} X_{-} & \text{and} &&
        \Xr_{+} & = W^{\trans} X_{+},
    \end{aligned}
  \end{equation}
  then the controller $K = \Kr W^{\trans}$ stabilizes the
  linearized system and is also locally stabilizing the steady state
  $(\xb, \ub)$ of the nonlinear system in the sense
  of \Cref{prp:stabnl}.
\end{theorem}

\begin{proof}[Proof of \Cref{thm:partstabdata}.]
  Let $(A, B)$ be the linear model from which the data $\datatrip$ has been
  sampled.
  Because the columns of $W$ span the eigenspace of unstable eigenvalues, it
  holds that
  \begin{equation} \label{eqn:equipartstabTmp1}
    \begin{aligned}
      W^{\trans} A \left( W^{\dagger} \right)^{\trans} & = \Ar, & 
      W^{\trans} B & = \Br & \text{and} &&
      W^{\trans} X_{0} & = \Xr_{0},
    \end{aligned}
  \end{equation}
  where the spectrum of $\Ar$ is the unstable spectrum of $A$ that contributes
  to the system dynamics and $X_{0}$ is a basis of the subspace of initial
  conditions $\Xcal_{0}$.
  For the reduced data in~\cref{eqn:datatrunc} it holds that
  \begin{equation} \label{eqn:equipartstabTmp2}
    \Xr_{+} = W^{\trans} X_{+} = W^{\trans} A X_{-} + W^{\trans} B U_{-}
      = \Ar W^{\trans} X_{-} + \Br U_{-}
      = \Ar \Xr_{-} + \Br U_{-}.
  \end{equation}
  Thus, $(\Ar, \Br)$ is a model that explains the data $\datatripr$.
  Since the nonlinear system is assumed to be stabilizable, this holds
  for the linearized model, too.
  Since $\datatripr$ from~\cref{eqn:datatrunc} is informative for stabilization,
  there is a $\Kr$ such that $\Sigmar_{\rm{i/s}} \subseteq \Sigmar_{\Kr}$.
  Together with~\cref{eqn:equipartstabTmp2}, the feedback $\Kr$ stabilizes also
  the linear model $(\Ar, \Br)$ from~\cref{eqn:equipartstabTmp1}.
  The unstable eigenvalues of the high-dimensional $A$ that do not contribute
  to the system dynamics can be replaced by stable eigenvalues as it does not
  change the dynamics and consequently the given data; see, e.g.,~\cite{WerP22}.
  Therefore, one can assume $(A, B)$ to be stabilizable.
  From \Cref{thm:partstableft} it follows that $K = \Kr W^{\trans}$ stabilizes
  $(A, B)$, which means it stabilizes the linearized system.
  Since $K$ stabilizes a linearized system of a nonlinear process about the
  steady state $(\xb,\ub)$, it follows from \Cref{prp:stabnl} that it also
  stabilizes the nonlinear process locally at $(\xb,\ub)$, which concludes the
  proof.
\end{proof}

It has been shown in~\cite[Cor.~3 and~4]{WerP22} that the minimum number of
state observations required for
stabilizing all low-dimensional linear systems explaining a given data triplet
scales with the intrinsic, minimal state-space dimension of the underlying
system that is described by the model from which data were obtained.
Building on the data reduction in~\cref{eqn:datatrunc} in
\Cref{thm:partstabdata}, the following corollary shows that the restriction to
the unstable dynamics in $\Wcal$ 
can further reduce the number of required state observations
to $\neigr$, the dimension of $\Wcal$.

\begin{corollary}%
  \label{cor:numsampunstab}
  Let \Cref{thm:partstabdata} apply and let $\neigr$ be the number of excitable
  unstable eigenvalues of the linear model corresponding to the nonlinear
  system.
  Then, there always exist $\neigr$ states of the linear model that are
  sufficient to be observed for inferring a guaranteed locally stabilizing
  controller for the nonlinear system, even if high-dimensional states of
  dimension $\nh \gg \neigr$ are observed.
\end{corollary}

The number of state observations $\neigr$ in \Cref{cor:numsampunstab}
for constructing a guaranteed stabilizing controller for the nonlinear system
does neither scale with the large state dimension $\nh$ nor with the intrinsic
system dimension $\nmin$.
In fact, $\neigr$ is typically orders of magnitude smaller than $\nh$ and
$\nmin$. 
However, there is a certain price that needs to be paid for the application of
\Cref{thm:partstabdata,cor:numsampunstab}:
A basis matrix $W$ (or alternatively $Q_{\perp}$ from \Cref{lmm:orthcomp}) must
be available.
Approaches for the computation of $W$ (or $Q_{\perp}$) from data are discussed
in \Cref{sec:eigenspaces}.


\subsection{Computational approach for controller inference}%
\label{subsec:partstabalg}

This section introduces a computational approach for context-aware controller
inference that is motivated by the theoretical considerations of the previous
sections.
We now consider data $\datatripn$ from the nonlinear system rather
than idealized data $\datatrip$ as in the previous section.
The computational procedure is summarized in \Cref{alg:partstab}.

\begin{algorithm}[t]
  \SetAlgoHangIndent{1pt}
  \DontPrintSemicolon
  \caption{Context-aware controller inference}
  \label{alg:partstab}
  
  \KwIn{Data triplet $\datatripn$, steady state $(\xb,\ub)$,
    basis matrix $\Wt \in \R^{\nh \times \neig}$ of the left eigenspace $\Wcal$
    corresponding to the unstable eigenvalues of the linear model.}
  \KwOut{Controller $K$.}
  
  Shift $\datatripn$ by the steady state $(\xb, \ub)$ to obtain $\datatripg$
    via~\cref{eqn:datatransct} (or~\cref{eqn:datatransdt}).\;
  \label{alg:partstab_shift}
  
  Project onto the left eigenspace $\Wcalt$ so that only the unstable
    dynamics in the data remain
    \begin{equation*}
      \begin{aligned}
        \Xt_{-} & = \Wt^{\trans} \Xg_{-} & \text{and} &&
          \Xt_{+} & = \Wt^{\trans} \Xg_{+}.
      \end{aligned}
    \end{equation*}
    \vspace{-\baselineskip}\;
  \label{alg:partstab_proj}
    
  Derive the basis matrix $V \in \R^{\neig \times \neigr}$ of $\Vcal$ using the
    singular value decomposition
    \begin{equation*}
      \begin{bmatrix} \Xt_{-} & \Xt_{+} \end{bmatrix} =
        \begin{bmatrix} V & V_{2} \end{bmatrix}
        \begin{bmatrix} \Sigma_{1} & 0 \\ 0 & 0 \end{bmatrix}
        U^{\trans}.\;
    \end{equation*}
  \label{alg:partstab_svd}
  
  Project onto $\Vcal$ so that only unstable and excitable dynamics remain
    \begin{equation*}
      \begin{aligned}
        \Xr_{-} & = V^{\trans} \Xt_{-} & \text{and} &&
          \Xr_{+} & = V^{\trans} \Xt_{+}.
      \end{aligned}
    \end{equation*}
    \vspace{-\baselineskip}\;
  \label{alg:partstab_projV}
  
  Compute the eigenbasis matrix $W = \Wt V$ of the minimal subspace of unstable
  dynamics.\;
  \label{alg:partstab_WV}

  Infer a low-dimensional stabilizing controller $\Kr = \Ug_{-}
    \Theta (\Xr_{-} \Theta)^{-1}$ from $\Ug_{-}, \Xr_{-},$ and  $\Xr_{+}$ by
    solving~\cref{eqn:ctlmi} (or~\cref{eqn:dtlmi}) for the unknown $\Theta$.\;
  \label{alg:partstab_infer}

  Lift the inferred low-dimensional controller $\Kr$ to 
    \begin{equation*}
      K = \Kr W^{\trans}.
    \end{equation*}
    \vspace{-\baselineskip}\;
  \label{alg:partstab_lift}
\end{algorithm}

In Step~\ref{alg:partstab_shift} of \Cref{alg:partstab}, the data $\datatripn$
from the nonlinear system are shifted as
\begin{equation} \label{eqn:datatransct}
  \begin{aligned}
    \Ug_{-} & = \begin{bmatrix} \un(t_{0}) - \ub & \ldots &
      \un(t_{\dT - 1}) - \ub \end{bmatrix},\quad
    \Xg_{-}  = \begin{bmatrix} \xn(t_{0}) - \xb &  \ldots &
      \xn(t_{\dT - 1}) - \xb \end{bmatrix},\\
    \Xg_{+} & = \begin{bmatrix} \dxn(t_{0}) &  \ldots &
      \dxn(t_{\dT - 1}) \end{bmatrix},
  \end{aligned}
\end{equation}
in the continuous-time case and 
\begin{equation} \label{eqn:datatransdt}
  \begin{aligned}
    \Ug_{-} & = \begin{bmatrix} \un(0) - \ub & \ldots &
      \un(\dT - 1) - \ub \end{bmatrix},\quad
    \Xg_{-}  = \begin{bmatrix} \xn(0) - \xb & \ldots &
      \xn(T - 1) - \xb \end{bmatrix}, \\
    \Xg_{+} & = \begin{bmatrix} \xn(1) - \xb  & \ldots &
      \xn(T) - \xb \end{bmatrix},
  \end{aligned}
\end{equation}
in the discrete-time case. The shifted
data~\cref{eqn:datatransct,eqn:datatransdt} are related to the idealized data
$\datatrip$, which are used in
\Cref{subsec:partstabdata}, as 
\begin{equation*}
 U_{-}  = \Ug_{-},\quad
    X_{-}  = \Xg_{-} + \mathcal{O}\left( (\Xn_{-} - \xb)^{2} \right),\quad
    X_{+}  = \Xg_{+} + \mathcal{O}\left( (\Xn_{+} - \xb)^{2} \right) +
      \mathcal{O}\left( (\Un_{+} - \ub)^{2} \right),
\end{equation*}
which holds because of the Taylor approximation~\cref{eqn:taylor} and shows that
data from the nonlinear system are close to the idealized data if
trajectories stay within a small neighborhood of $(\xb, \ub)$. 

Step~\ref{alg:partstab_proj} of \Cref{alg:partstab} projects the shifted data
onto the subspace of unstable dynamics $\Wcal$ so that only the unstable parts
of the trajectories remain in $\Xt_{-}$ and $\Xt_{+}$.
The basis matrix $\Wt$ may not span the minimal subspace of unstable dynamics as
required by \Cref{thm:partstabdata}, i.e., the subspace $\Wcalt$ contains an
eigenspace of $A$ corresponding to zero dynamics; cf. the third block row
in~\cref{eqn:kalman}.
This can be the case when $\Wt$ is computed as shown in the next section.
The zero dynamics lead to a smaller minimal system dimension $\neigr < \neig$
and rank deficient data matrices $\Xt_{-}$ and $\Xt_{+}$.
Any rank revealing decomposition of the data can be used to remove their kernel
and reduce the eigenspace to the requested minimal subspace.
In Step~\ref{alg:partstab_svd}, the singular value decomposition is used
for this purpose.
Then the data is projected onto the minimal subspace of unstable dynamics
in Step~\ref{alg:partstab_projV} and the corresponding basis matrix is computed
in Step~\ref{alg:partstab_WV}.

Step~\ref{alg:partstab_infer} of \Cref{alg:partstab} infers a controller from
the unstable and excitable dynamics in $\Xr_{-}$ and $\Xr_{+}$.
The controller is given by $\Kr = \Ug_{-} \Theta (X_{-} \Theta)^{-1}$, where
$\Theta \in \R^{T \times \neigr}$ solves, in the continuous-time case,
\begin{equation} \label{eqn:ctlmi}
  \begin{aligned}
    \Xr_{-} \Theta & > 0 & \text{and} &&
    \Xr_{+} \Theta + \Theta^{\trans} \Xr_{+}^{\trans} < 0
  \end{aligned}
\end{equation}
and in the discrete-time case
\begin{equation} \label{eqn:dtlmi}
  \begin{aligned}
    \Xr_{-} \Theta & = (\Xr_{-} \Theta)^{\trans} & \text{and} &&
      \begin{bmatrix} \Xr_{-} \Theta & \Xr_{+} \Theta \\
      (\Xr{+} \Theta)^{\trans} & \Xr_{-} \Theta \end{bmatrix} > 0;
  \end{aligned}
\end{equation}
see~\cite{WerP22,VanETetal20,DePT20}.

In the last step of \Cref{alg:partstab}, the inferred controller $\Kr$ is lifted
to the original state-space dimension $K = \Kr W^{\trans}$ such that
$\un(t) = K (\xn(t) - \xb(t)) + \ub(t)$ stabilizes the nonlinear
system under the assumptions given in \Cref{prp:stabnl}.

In \Cref{alg:partstab}, the stabilizing controller is inferred from $\Xr_{-}$
and $\Xr_{+}$ by solving the linear matrix
inequalities~\cref{eqn:ctlmi,eqn:dtlmi}, respectively.
If the number of observed states $\dT$ equals
the dimension $\neigr$ of the unstable dynamics, then a computationally more
efficient alternative to solving~\cref{eqn:ctlmi,eqn:dtlmi} is introduced
in~\cite[Thm.~16]{VanETetal20} and~\cite[Prp.~2 and Cor.~1]{WerP22}.
It computes only the inverse of the reduced data matrix $\Xr_{-}$.
Also, to identify a model that describes the unstable system
dynamics, Step~\ref{alg:partstab_infer} can be replaced
by~\cite[Prp.~1]{WerP22}.
With the identified model describing the unstable dynamics, classical
stabilization approaches can be employed to compute a suitable $\Kr$.

Note that the system needs to be stabilizable for the
matrix inequalities~\cref{eqn:ctlmi,eqn:dtlmi} to be solvable.
If the system is not stabilizable, the inequalities have no solution and the
solvers employed in Step~\ref{alg:partstab_infer} of \Cref{alg:partstab} throw
an error that no stabilizing controller can be inferred.


\section{Estimating left eigenspaces from gradient samples}%
\label{sec:eigenspaces}

In this section, we discuss leveraging samples of gradients to estimate a basis
matrix $W$ of the left eigenspace $\Wcal$ of unstable eigenvalues.
Let
\begin{equation} \label{eqn:fimplicit}
  F(\xn(t), \un(t)) := \dxn(t) - f\big(\xn(t), \un(t)\big)
\end{equation}
denote the residual formulation of the continuous-time model~\cref{eqn:ctnsys}.
The transposed gradient of~\cref{eqn:fimplicit} at the steady state $(\xb, \ub)$
is
\begin{equation} \label{eqn:adjointsteady}
  \Fcal[\xb, \ub] := \Big(\partial_{\rm{x}} F\big(\xb, \ub\big)\Big)^{\trans}
    = \left(\partial_{\rm{x}}\frac{\rm{d}}{\rm{d} t}
    \xb\right)^{\trans}
    - \big(\partial_{\rm{x}} f(\xb, \ub)\big)^{\trans}
  = -A^{\trans},
  \end{equation}
which can be applied to a vector $x$ as
\begin{equation}
  \label{eqn:adjointop}
  \mathcal{F}[\xb, \ub](x) := -A^{\trans} x,
\end{equation}
where $A$ is the state matrix of the linearized model from~\cref{eqn:ctsys}.
A basis of the left eigenspace $\Wcal$ of $A$ can be
obtained from samples of~\cref{eqn:adjointop}.
The same holds for the discrete-time case~\cref{eqn:dtnsys}.

The transposed gradient of the residual~\cref{eqn:fimplicit} occurs, for
example, in the adjoint equation of~\cref{eqn:fimplicit}.
For some cost functional $J$ that depends on the solution trajectory
$(\xn(t))_{t \geq 0}$ and the input
$(\un(t))_{t \geq 0}$, the adjoint equation is
\begin{equation} \label{eqn:adjoint}
  \Big(\partial_{\rm{x}} F\big(\xn(t), \un(t)\big)\Big)^{\trans}
    \lambda(t) = \left(\partial_{\rm{x}} J\big(\xn(t), \un(t)\big)
    \right)^{\trans},
\end{equation}
where $\lambda(t) \in \R^{\nh}$ are the adjoint variables and
$\partial_{\rm{x}}$ denotes the partial derivative with respect to $\xn(t)$;
see, e.g.,~\cite[Sec.~2.1]{FarHFetal13}.
The matrix on the left-hand side of~\cref{eqn:adjoint} is the
linear adjoint operator of~\cref{eqn:fimplicit}, which describes the
linearized dynamics about the solution $\xn(t)$ backwards in time.
From~\cref{eqn:adjointsteady,eqn:adjoint} it follows that the solution of the
adjoint equation at the steady state of interest amounts to the evaluation of
the negative, transposed state matrix of the linear
model~\cref{eqn:adjointsteady} that has been considered for the theory in
\Cref{sec:inferpartstab}.
In this setting, the resulting matrix in~\cref{eqn:adjointsteady} is called the
adjoint operator, which can be applied to tangent directions $x \in \C^{\nh}$
as shown in~\cref{eqn:adjointop}.


\subsection{Active estimation of left eigenspaces}%
\label{subsubsec:krylov}

It is a common situation in major software packages, for example, in FEniCS and
Firedrake~\cite{MitFD19}, that the function~\cref{eqn:adjointop} is
provided by the ability to solve the adjoint equation~\cref{eqn:adjoint}.
Thus, models that are implemented in such software packages allow
querying~\cref{eqn:adjointop} in a black-box fashion.
Similarly, models of nonlinear systems where the right-hand side function $f$
in~\cref{eqn:ctnsys} (or~\cref{eqn:dtnsys}) is implemented with automatic
differentiation also typically provide a vector-Jacobian product routine
\texttt{vjp} that computes a query of the adjoint operator into given
directions~\cref{eqn:adjointop}; see~\cite{BraFHetal18, FroJL18}.
Motivated by these capabilities, we consider in this section the situation that
we are allowed to sample the transposed gradient in~\cref{eqn:adjointop} at
tangent directions $x \in \C^{\nh}$.

With the function $\Fcal[\xb, \ub]$ from~\cref{eqn:adjointop} available, Krylov
subspace methods can be directly applied to compute eigenvalues and eigenvectors
of $-A^{\trans}$, without having to assemble $A$ or $-A^{\trans}$; see, for
example,~\cite{Ste01,GolV13}.
Krylov subspace methods converge first to eigenvalues with large absolute
values.
In the discrete-time case, the unstable eigenvalues, which are the ones we are
interested in, have larger absolute values than the stable eigenvalues.
Thus, in the discrete-time case, Krylov subspace methods first converge to the
eigenvalues and eigenvectors that we need for controller inference.

In the continuous-time case, shift-and-invert operators can be applied to make
Krylov methods converge quickly to the unstable eigenvalues~\cite{Ruh84}.
The idea is to shift the considered linear operator~\cref{eqn:adjointop} with a
given $\sigma \in \C$ and invert the shifted operator such that the eigenvalues
closest to $\sigma$ in the complex plane become those with largest absolute
value to which the Krylov methods converge first.
Given a suitable $\sigma \in \C$, (rational) Krylov subspace methods, like
the conjugate gradient method and GMRES~\cite[Chap.~10]{GolV13}, solve linear
systems of the form
\begin{equation} \label{eqn:linsolve}
  -\Fcal[\xb, \ub](x) - \sigma x = b,
\end{equation}
for the unknown vector $x \in \C^{\nh}$ and a given right-hand side
$b \in \C^{\nh}$.


\subsection{Estimation of left eigenspaces from gradient samples}%
\label{subsubsec:eigendata}

Consider now the situation that we have a discrete-time system and data from the
operator~\cref{eqn:adjointop} such that the column-wise evaluation
$\Fcal[\xb, \ub](X_{-}^{\rm{a}}) = X_{+}^{\rm{a}}$ holds, where
\begin{equation} \label{eqn:adjointdata}
  \begin{aligned}
    X^{\rm{a}}_{-} & = \begin{bmatrix} x_{0}^{\rm{a}} & \ldots &
      x_{\dT-1}^{\rm{a}} \end{bmatrix} & \text{and} &&
    X^{\rm{a}}_{+} & = \begin{bmatrix} x_{1}^{\rm{a}} & \ldots &
      x_{\dT}^{\rm{a}} \end{bmatrix}.
  \end{aligned}
\end{equation}
This is a typical situation when the adjoint equation~\cref{eqn:adjoint} is 
solved by an iterative method.
The estimation of the eigenbasis of interest from such data is possible using
the dynamic mode decomposition~\cite{Sch10, TuRLetal14}.
The eigenvalues and eigenvectors of
\begin{equation} \label{eqn:adjointdmd}
  H_{\rm{a}} = X^{\rm{a}}_{+} \Vc \Sigmac^{-1} \Uc^{\trans},
\end{equation}
with the economy-size singular value decomposition
$X^{\rm{a}}_{-} = \Uc \Sigmac \Vc$, are approximations to the eigenvalues and
eigenvectors of $-A^{\trans}$; see~\cite[Sec.~2]{TuRLetal14}.
Note that computations with the large-scale dense
matrix~\cref{eqn:adjointdmd} to obtain the eigenvalues and eigenvectors can be
avoided via a low-dimensional reformulation of~\cref{eqn:adjointdmd};
see~\cite{ProBK16}.

If the data in~\cref{eqn:adjointdata} is generated as recursive sequence, i.e., 
$x_{k+1}^{\rm{a}} = \Fcal[\xb, \ub](x_{k}^{\rm{a}})$ for
$k = 0, 1, \ldots, \dT - 1$,
the theory about Krylov subspace methods applies and the eigenvalues
converge first to those of $-A^{\trans}$ with largest magnitude, i.e.,
the requested unstable eigenvalues~\cite{MisP29}.
For the continuous-time case, such kind of results about eigenvalue estimation
are unknown as of now to the best of the authors' knowledge.


\section{Numerical examples}%
\label{sec:examples}

The experiments have been run on 
an Intel(R) Core(TM) i7-8700 CPU at 3.20\,GHz with 16\,GB main memory.
The algorithms are implemented in MATLAB 9.9.0.1467703 (R2020b) on CentOS Linux
release 7.9.2009 (Core).
For the comparison with reinforcement learning, we use the implementations from
the Reinforcement Learning Toolbox version 1.3.
For solving the linear matrix inequalities~\cref{eqn:ctlmi,eqn:dtlmi},
the disciplined convex programming toolbox CVX version 2.2, build 1148
(62bfcca)~\cite{GraB08, GraB20} is used together with
MOSEK version 9.1.9~\cite{MOS19} as inner optimizer.
Source code, data, and numerical results are available at~\cite{supWer23}.


\subsection{Experimental setup}%
\label{subsec:setup}


\subsubsection{Simulations, data generation and tests}%
\label{subsubsec:simulations}

Data are generated with realizations of normally distributed inputs.
Bases of the unstable dynamics (left eigenbases of unstable eigenvalues) are
estimated from adjoint operators with MATLAB's \texttt{eigs}, which employs the
Krylov-Schur method~\cite{Ste01} as suggested in \Cref{subsubsec:krylov}.
In case of linear system dynamics, we assume zero initial conditions and a test
input is chosen as unit step, which means that the states converge to a constant
vector if the controller stabilizes the system.
In case of nonlinear dynamics, the initial condition is the steady state of
interest to which an input disturbance is applied at time 0.
The states will then converge to the
steady state, if the controller stabilizes the system.

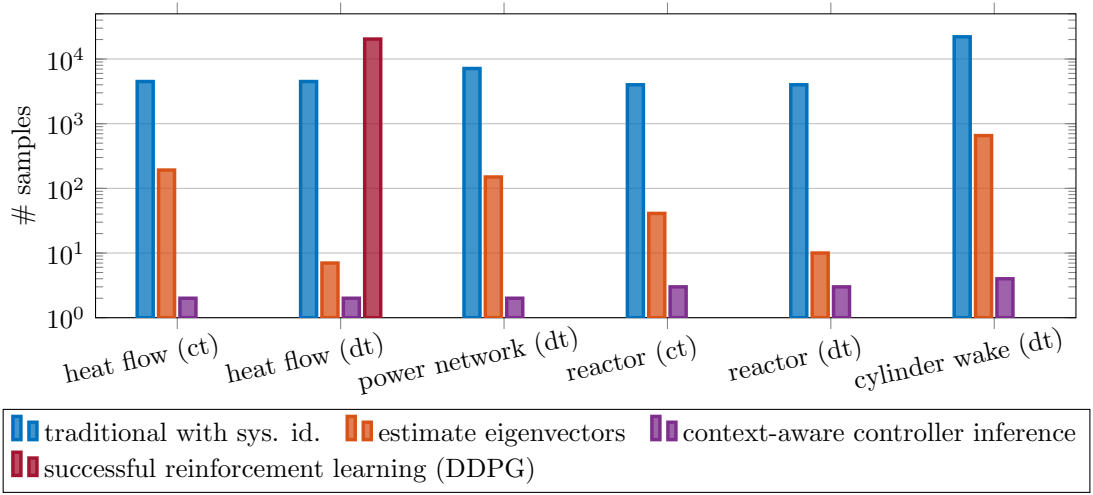
\begin{figure}[t]
  \centering
  \tikzexternalenable%
  \tikzsetnextfilename{overview_samples}%
  \begin{tikzpicture}[font = \plotfontsize]
  \begin{axis}[
    width  = .86\textwidth,
    height = .17\textheight,
    scale only axis,
    ybar,
    bar width    = 6pt,
    ymode        = log,
    ymin         = 1,
    ymax         = 5e+04,
    ylabel       = {\#~samples},
    ymajorgrids,
    ylabel style = {yshift = -.3em},
    xtick             = data,
    symbolic x coords = {heat~flow~(ct), heat~flow~(dt), power~network~(dt),
      reactor~(ct), reactor~(dt), cylinder~wake~(dt)},
    xticklabel style  = {rotate = 12, anchor = north, xshift = -1.4em},
    legend columns    = 3,
    legend cell align = {left},
    legend style      = {
      at     = {(.46,-0.3)},
      anchor = north,
      /tikz/every even column/.append style = {column sep = 0.25cm}},
    cycle list name = samplelist
  ]
  
    \addplot+[ybar, bar shift = {-12pt}] coordinates{
      (heat~flow~(ct), 4491)
      (heat~flow~(dt), 4491)
      (power~network~(dt), 7139)
      (reactor~(ct), 4000)
      (reactor~(dt), 4000)
      (cylinder~wake~(dt), 22066)
    };
    \addlegendentry{traditional with sys. id.}
  
    \addplot+[ybar, bar shift = {-4pt}] coordinates{
      (heat~flow~(ct), 192)
      (heat~flow~(dt), 7)
      (power~network~(dt), 150)
      (reactor~(ct), 41)
      (reactor~(dt), 10)
      (cylinder~wake~(dt), 655)
    };
    \addlegendentry{estimate eigenvectors}
  
    \addplot+[ybar, bar shift = {4pt}] coordinates{
      (heat~flow~(ct), 2)
      (heat~flow~(dt), 2)
      (power~network~(dt), 2)
      (reactor~(ct), 3)
      (reactor~(dt), 3)
      (cylinder~wake~(dt), 4)
    };
    \addlegendentry{context-aware controller inference}
  
    \addplot+[ybar, bar shift = {12pt}] coordinates{
      (heat~flow~(dt), 20365)
    };
    \addlegendentry{\makebox[0pt][l]{successful reinforcement learning (DDPG)}}
  
    \addlegendimage{empty legend}
    \addlegendentry{}
    
    \addlegendimage{empty legend}
    \addlegendentry{}
  \end{axis}
\end{tikzpicture}%
  \tikzexternaldisable%

  \caption{Summary of numerical experiments:
    The proposed context-aware controller inference approach provides
    stabilizing controllers in all experiments with an at least one order of
    magnitude lower number of data samples (state observations +
    gradient samples) than traditional two-step
    approaches that first identify models and then construct controllers.
    Reinforcement learning via deep deterministic policy gradient (DDPG)
    manages to stabilize only one system in our experiments, for which it
    requires $2\,260$ times as many samples as the proposed context-aware
    controller inference.}
  \label{fig:overview_samples}
\end{figure}

\Cref{fig:overview_samples} provides an overview about the performance of
context-aware controller inference in the following numerical experiments.
In all experiments, context-aware controller inference requires at least one
order of magnitude fewer data samples to derive stabilizing controllers
than required to learn a model for the traditional two-step data-driven control
approach, including the number of gradient samples to estimate the bases of
unstable dynamics.
For the identification of a model that is guaranteed to capture the system
dynamics for the construction of controllers in general, at least
$T = \nh + \np$ samples are necessary, where $\nh$ is the dimension of the full
state space and $\np$ the number of inputs~\cite{WerP22}.
Under certain assumptions one can identify systems with fewer samples if
structure in the system dynamics is present; however, structure such as low
rankness depends on the underlying physics and typically still requires
subspaces of higher dimensions than the subspaces of unstable dynamics that are
used in our approach; see the introduction in \Cref{sec:intro}.


\subsubsection{Comparison with reinforcement learning with deep deterministic policy gradient}%
\label{subsubsec:ddpg}

We compare our approach in the discrete-time case with controllers obtained from
reinforcement learning via the deep deterministic policy gradient (DDPG)
method~\cite{SilLHetal14, LilHPetal15}, which can handle continuous
observation and action spaces.
The setup for the DDPG agents is the same as for all numerical examples up to
tolerances as tuning parameters.
The actor is a shallow network with a single layer and zero bias such that the 
resulting weight matrix corresponds to the control matrix
$K \in \R^{\np \times \nh}$ and can be implemented the same way for simulations
as the controller designed by our proposed method.
The critic takes the concatenation of state observations and
the control inputs and applies to it a quadratic activation function.
The result is then compressed via a fully connected layer into the Q-value.
The reward function is
\begin{equation} \label{eqn:reward}
  r(t, x) = \begin{cases}
    -\alpha \frac{x^{\trans} x}{\lambda_{\rm{d}} n_{\rm{s}}(t)}, &
      \text{if}~\lVert x \rVert_{2} > \lambda_{\rm{u}} \sqrt{\nh},\\
    \alpha \left( x^{\trans} x \right) n_{\rm{s}}(t), &
      \text{if controller is stabilizing},\\
    -x^{\trans} x, & \text{otherwise},
  \end{cases}
\end{equation}
with the sum of steps $n_{\rm{s}}(t) = 0.5 t (t + 1)$ and problem-dependent
regularization parameters $\alpha = 100 \sqrt{\nh}$, $\lambda_{\rm{d}}$ and
$\lambda_{\rm{u}}$.
The reward function~\cref{eqn:reward} is motivated by the linear-quadratic
regulator design~\cite{Sim96} to penalize the deviation of the current state
from $0$; see, for example,~\cite{PerUS21} where a similar but continuous
reward function has been used.
The first branch in~\cref{eqn:reward} evaluates to a large negative reward
if the norm of the state becomes too large, which is an indication for
destabilization, after which the current training episode is terminated.
This is necessary to cope with the limits of numerical accuracy during the
training.
The second branch in~\cref{eqn:reward} gives a positive reward to end the
training with a stabilizing controller.
Note that it is not possible to check if a given controller stabilizes the
underlying system without system identification; cf.~\cite{VanETetal20,WerP22}.
We therefore fall back to a heuristic and check if the states of three
consecutive time steps are close to $0$ and then return a positive reward that
outweighs the previously accumulated negative rewards to end the training.
The training itself is performed using discrete-time simulations of the
linear models over the same time period as considered for the test
simulations and starting at normally distributed, scaled random initial
conditions.
We have set a maximum of $1\,000$ training episodes.
More details about the setup can be found in the supplementary
material~\cite{supWer23}.


\subsection{Experiments with linear systems}%
\label{subsec:lexamples}

In this section, we consider two examples with linear dynamics such that no
additional linearization or data shift in the sense
of~\cref{eqn:datatransct,eqn:datatransdt} is necessary.


\subsubsection{Disturbed heat flow}%
\label{subsubsec:hf2d5}

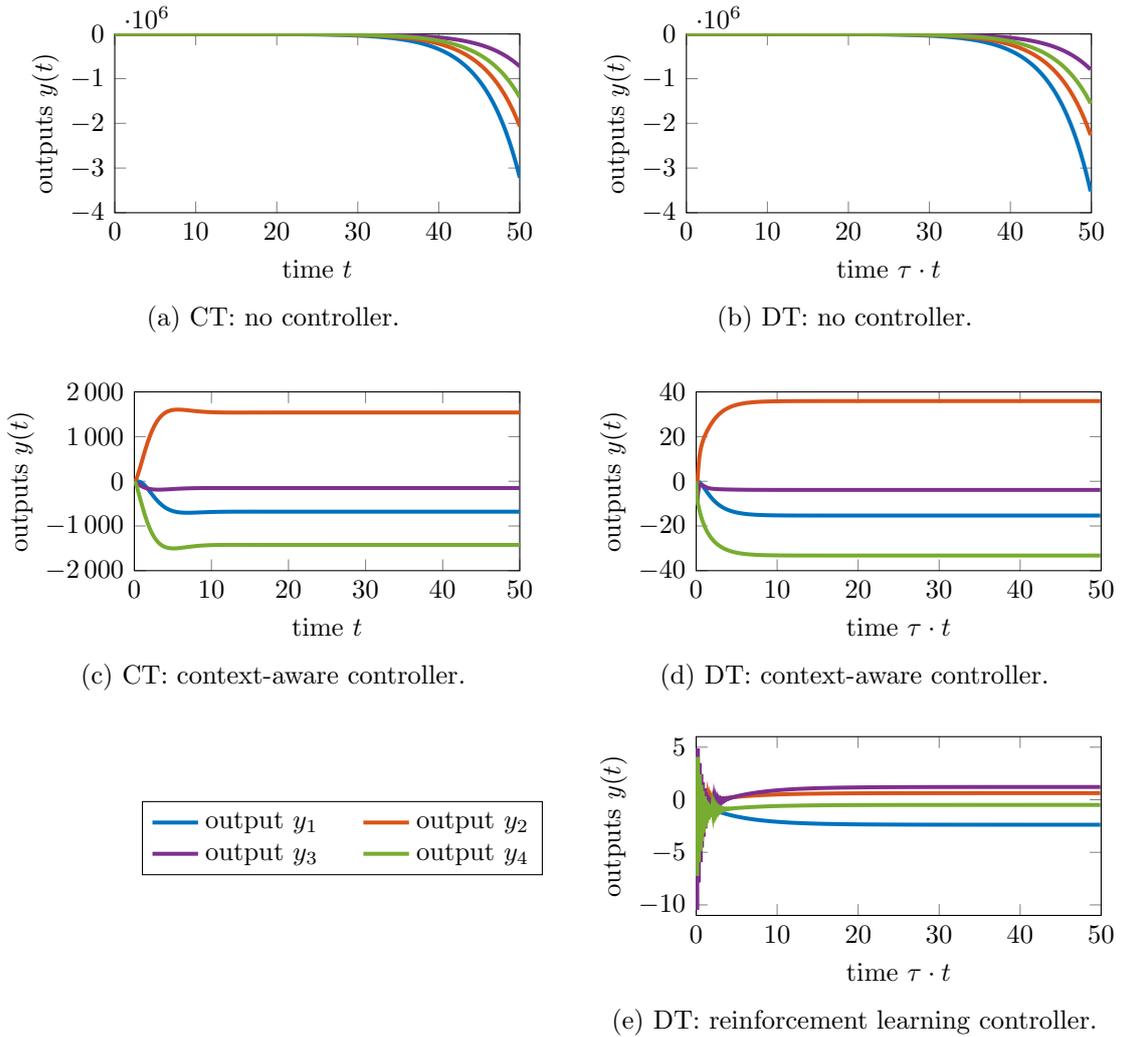
\begin{figure}[t]
  \centering%
  \begin{subfigure}[b]{.49\linewidth}
    \raggedleft
  \tikzexternalenable%
  \tikzsetnextfilename{hf2d5_ct_sim_nofb}%
  \begin{tikzpicture}[font = \plotfontsize]
  \pgfplotstableread{graphics/data/hf2d5_ct_sim_nofb.dat}\tableSIM
  
  \begin{axis}[%
    name   = states,
    width  = .725\textwidth,
    height = .1\textheight,
    scale only axis,
    xmin = 0,
    xmax = 50,
    ymin = -4e+06,
    ymax = 0,
    xminorticks = false,
    yminorticks = false,
    xlabel = {time $t$},
    ylabel = {outputs $y(t)$},
    ylabel style   = {yshift = -.3em},
    scaled x ticks = false,
    x tick label style = {/pgf/number format/1000 sep={\,}},
    y tick label style = {/pgf/number format/1000 sep={\,}},
    cycle list name    = stateslist
  ]
  
    \foreach \y in {3, 4, ..., 6}{
      \addplot+ table[x index = 0, y index = \y] {\tableSIM};
    }
  \end{axis}
  \node[minimum size = 1cm, inner sep = 0pt, outer sep = 0pt]
    at (states.north) {};
\end{tikzpicture}%
  \tikzexternaldisable%

    \caption{CT: no controller.}
    \label{fig:hf2d5_a}
  \end{subfigure}%
  \hfill%
  \begin{subfigure}[b]{.49\linewidth}
    \raggedleft
  \tikzexternalenable%
  \tikzsetnextfilename{hf2d5_dt_sim_nofb}%
  \begin{tikzpicture}[font = \plotfontsize]
  \pgfplotstableread{graphics/data/hf2d5_dt_sim_nofb.dat}\tableSIM
  
  \begin{axis}[%
    name   = states,
    width  = .725\textwidth,
    height = .1\textheight,
    scale only axis,
    xmin = 0,
    xmax = 50,
    ymin = -4e+06,
    ymax = 0,
    xminorticks = false,
    yminorticks = false,
    xlabel = {time $\tau \cdot t$},
    ylabel = {outputs $y(t)$},
    ylabel style   = {yshift = -.3em},
    scaled x ticks = false,
    x tick label style = {/pgf/number format/1000 sep={\,}},
    y tick label style = {/pgf/number format/1000 sep={\,}},
    cycle list name    = stateslist
  ]
  
    \foreach \y in {3, 4, ..., 6}{
      \addplot+ table[x index = 0, y index = \y] {\tableSIM};
    }
  \end{axis}
  \node[minimum size = 1cm, inner sep = 0pt, outer sep = 0pt]
    at (states.north) {};
\end{tikzpicture}%
  \tikzexternaldisable%

    \caption{DT: no controller.}
    \label{fig:hf2d5_b}
  \end{subfigure}
  \vspace{1\baselineskip}
  
  \begin{subfigure}[b]{.49\linewidth}
    \raggedleft
  \tikzexternalenable%
  \tikzsetnextfilename{hf2d5_ct_sim_inferfb}%
  \begin{tikzpicture}[font = \plotfontsize]
  \pgfplotstableread{graphics/data/hf2d5_ct_sim_inferfb.dat}\tableSIM
  
  \begin{axis}[%
    name   = states,
    width  = .69\textwidth,
    height = .1\textheight,
    scale only axis,
    xmin = 0,
    xmax = 50,
    ymin = -2000,
    ymax = 2000,
    xminorticks = false,
    yminorticks = false,
    xlabel = {time $t$},
    ylabel = {outputs $y(t)$},
    ylabel style   = {yshift = -.3em},
    scaled x ticks = false,
    x tick label style = {/pgf/number format/1000 sep={\,}},
    y tick label style = {/pgf/number format/1000 sep={\,}},
    cycle list name    = stateslist
  ]
  
    \foreach \y in {3, 4, ..., 6}{
      \addplot+ table[x index = 0, y index = \y] {\tableSIM};
    }
  \end{axis}
\end{tikzpicture}%
  \tikzexternaldisable%

    \caption{CT: context-aware controller.}
    \label{fig:hf2d5_c}
  \end{subfigure}%
  \hfill%
  \begin{subfigure}[b]{.49\linewidth}
    \raggedleft
  \tikzexternalenable%
  \tikzsetnextfilename{hf2d5_dt_sim_inferfb}%
  \begin{tikzpicture}[font = \plotfontsize]
  \pgfplotstableread{graphics/data/hf2d5_dt_sim_inferfb.dat}\tableSIM
  
  \begin{axis}[%
    name   = states,
    width  = .725\textwidth,
    height = .1\textheight,
    scale only axis,
    xmin = 0,
    xmax = 50,
    ymin = -40,
    ymax = 40,
    xminorticks = false,
    yminorticks = false,
    xlabel = {time $\tau \cdot t$},
    ylabel = {outputs $y(t)$},
    ylabel style   = {yshift = -.3em},
    scaled x ticks = false,
    x tick label style = {/pgf/number format/1000 sep={\,}},
    y tick label style = {/pgf/number format/1000 sep={\,}},
    cycle list name    = stateslist
  ]
  
    \foreach \y in {3, 4, ..., 6}{
      \addplot+ table[x index = 0, y index = \y] {\tableSIM};
    }
  \end{axis}
\end{tikzpicture}%
  \tikzexternaldisable%

    \caption{DT: context-aware controller.}
    \label{fig:hf2d5_d}
  \end{subfigure}
  
  \begin{subfigure}[t]{.49\linewidth}
    \raggedleft
    \vspace{2\baselineskip}
  \tikzexternalenable%
  \tikzsetnextfilename{hf2d5_legend}%
  \begin{tikzpicture}[font = \plotfontsize]
  \begin{axis}[%
    hide axis,
    width  = 1cm,
    height = 1cm,
    scale only axis,
    xmin = 0,
    xmax = 10,
    ymin = 0.5,
    ymax = 1.5,
    legend columns    = 2,
    legend cell align = {left},
    legend style      = {
      at     = {(0,0)},
      anchor = center,
      /tikz/every even column/.append style = {column sep = 0.5cm}}
  ]
    
    \pgfplotsset{cycle list name = stateslist}
    \pgfplotsinvokeforeach{1, 2, ..., 4}{\addplot coordinates {(0,0)};}
    \addlegendentry{output $y_{1}$}
    \addlegendentry{output $y_{2}$}
    \addlegendentry{output $y_{3}$}
    \addlegendentry{output $y_{4}$}
  \end{axis}
\end{tikzpicture}%
  \tikzexternaldisable%

  \end{subfigure}%
  \hfill%
  \begin{subfigure}[t]{.49\linewidth}
    \raggedleft
    \vspace{0\baselineskip}
  \tikzexternalenable%
  \tikzsetnextfilename{hf2d5_dt_sim_learnfb}%
  \begin{tikzpicture}[font = \plotfontsize]
  \pgfplotstableread{graphics/data/hf2d5_dt_sim_learnfb.dat}\tableSIM
  
  \begin{axis}[%
    name   = states,
    width  = .725\textwidth,
    height = .1\textheight,
    scale only axis,
    xmin = 0,
    xmax = 50,
    ymin = -11,
    ymax = 6,
    xminorticks = false,
    yminorticks = false,
    xlabel = {time $\tau \cdot t$},
    ylabel = {outputs $y(t)$},
    ylabel style   = {yshift = -.3em},
    scaled x ticks = false,
    x tick label style = {/pgf/number format/1000 sep={\,}},
    y tick label style = {/pgf/number format/1000 sep={\,}},
    cycle list name    = stateslist
  ]
  
    \foreach \y in {3, 4, ..., 6}{
      \addplot+ table[x index = 0, y index = \y] {\tableSIM};
    }
  \end{axis}
\end{tikzpicture}%
  \tikzexternaldisable%

    \caption{DT: reinforcement learning controller.}
    \label{fig:hf2d5_e}
  \end{subfigure}

  \caption{Heat flow: The proposed context-aware controller inference 
    stabilizes the system in continuous and discrete time.
    By using $2\,260 \times$ more samples than context-aware controller
    inference, DDPG (reinforcement learning) also stabilizes the system, but it
    is more aggressive and leads to high-frequency oscillations at the
    beginning of the time interval.}
  \label{fig:hf2d5}
\end{figure}

\begin{figure}[t]
  \centering%
  \begin{subfigure}[b]{.49\linewidth}
    \raggedleft
  \tikzexternalenable%
  \tikzsetnextfilename{hf2d5_training}%
  \begin{tikzpicture}[font = \plotfontsize]
  \pgfplotstableread{graphics/data/hf2d5_dt_learnfb_train.dat}\tableTRAIN
  \pgfplotstablecreatecol[create col/expr={-log10(-\thisrow{1})}]{}\tableTRAIN
  
  \begin{axis}[%
    width  = .7\textwidth,
    height = .15\textheight,
    scale only axis,
    xmin = 0,
    xmax = 80,
    ymode = log,
    ymin = 1e-1,
    ymax = 1e+12,
    y dir       = reverse,
    ytick       = {1e+0, 1e+3, 1e+6, 1e+9, 1e+12},
    yticklabels = {$-10^{0}$, $-10^{3}$, $-10^{6}$, $-10^{9}$, $-10^{12}$},
    xminorticks = false,
    yminorticks = false,
    xlabel = {episode},
    ylabel = {accumulated reward},
    ylabel style   = {yshift = -.3em},
    scaled x ticks = false,
    x tick label style = {/pgf/number format/1000 sep={\,}},
    y tick label style = {/pgf/number format/1000 sep={\,}},
    cycle list name    = trainlist,
    cycle list shift   = 2
  ]
  
    \addplot+[const plot] table[x index = 0, y expr = -\thisrowno{1}]
      {\tableTRAIN};
  \end{axis}
\end{tikzpicture}%
  \tikzexternaldisable%

    \caption{DDPG training performance.}
    \label{fig:hf2d5_train_a}
  \end{subfigure}%
  \hfill%
  \begin{subfigure}[b]{.49\linewidth}
    \raggedleft
  \tikzexternalenable%
  \tikzsetnextfilename{hf2d5_rewards}%
  \begin{tikzpicture}[font = \plotfontsize]
  \pgfplotstableread{graphics/data/hf2d5_dt_rewards.dat}\tableREW
  
  \begin{loglogaxis}[%
    width  = .7\textwidth,
    height = .15\textheight,
    scale only axis,
    xmin = 1,
    xmax = 3e+4,
    ymin = 1e-1,
    ymax = 1e+12,
    y dir       = reverse,
    ytick       = {1e+0, 1e+3, 1e+6, 1e+9, 1e+12},
    yticklabels = {$-10^{0}$, $-10^{3}$, $-10^{6}$, $-10^{9}$, $-10^{12}$},
    xminorticks = false,
    yminorticks = false,
    xlabel = {\# samples / evaluations},
    ylabel = {episodal reward},
    ylabel style   = {yshift = -.3em},
    scaled x ticks = false,
    x tick label style = {/pgf/number format/1000 sep={\,}},
    y tick label style = {/pgf/number format/1000 sep={\,}},
    cycle list name    = trainlist
  ]
  
    \addplot+[const plot] table[x index = 0, y expr = -\thisrowno{1}] {\tableREW};
    \addplot+[const plot] table[x index = 0, y expr = -\thisrowno{2}] {\tableREW};
    \addplot+[const plot] table[x index = 0, y expr = -\thisrowno{3}] {\tableREW};
  \end{loglogaxis}
\end{tikzpicture}%
  \tikzexternaldisable%

    \caption{Reward comparison.}
    \label{fig:hf2d5_train_b}
  \end{subfigure}
  
  \vspace{.75\baselineskip}
  \tikzexternalenable%
  \tikzsetnextfilename{hf2d5_train_legend}%
  \begin{tikzpicture}[font = \plotfontsize]
  \begin{axis}[%
    hide axis,
    width  = 1cm,
    height = 1cm,
    scale only axis,
    xmin = 0,
    xmax = 10,
    ymin = 0.5,
    ymax = 1.5,
    legend columns    = 2,
    legend cell align = {left},
    legend style      = {
      at     = {(0,0)},
      anchor = center,
      /tikz/every even column/.append style = {column sep = 0.5cm}}
  ]
    
    \pgfplotsset{cycle list name = trainlist, cycle list shift = 2}
    \pgfplotsinvokeforeach{1, 2, 3}{\addplot coordinates {(0,0)};}
    \addlegendentry{reinforcement learning}
    \addlegendentry{context-aware control, excl. eigenbasis estimation}
    \addlegendentry{\makebox[0pt][l]{context-aware control, incl.
      eigenbasis estimation}}
  \end{axis}
\end{tikzpicture}%
  \tikzexternaldisable%

  \vspace{0\baselineskip}
  
  \caption{Heat flow:
    DDPG takes $77$ episodes to learn a controller that satisfies the
    stabilization test.
    This corresponds to $10^{4} \times$ more state observations 
    than context-aware controller inference if an estimate of the basis of the
    unstable dynamics is available a priori and $2\,260 \times$ more observations
    if a basis has to be estimated.}
  \label{fig:hf2d5_train}
\end{figure}
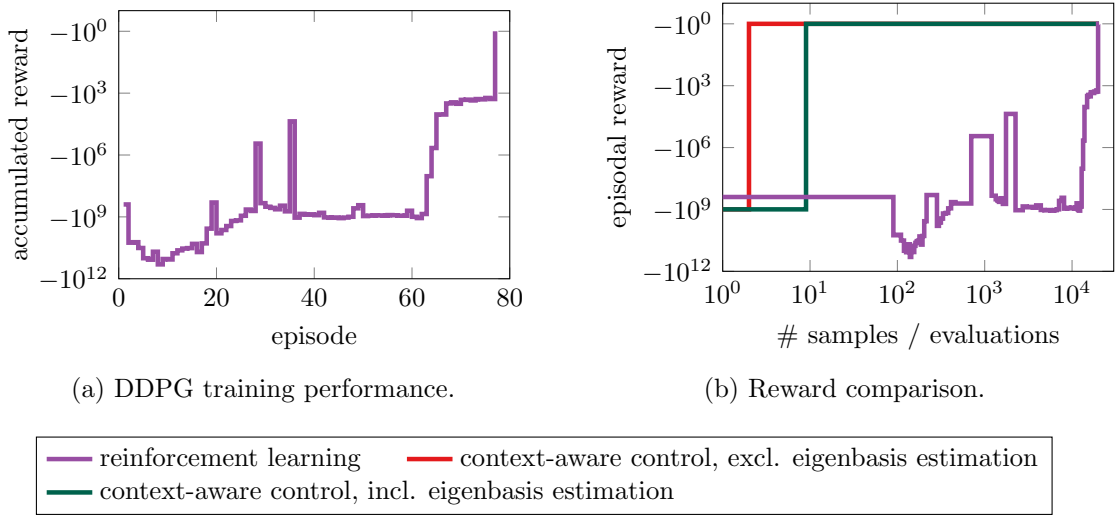

Consider a two-dimensional heat flow describing
the heating process in a rectangular domain affected by disturbances.
We use model \textsf{HF2D5} introduced in~\cite{Lei04} to describe the
underlying system.
The model is continuous in time with $\nh = 4\,489$ states
and $\np = 2$ inputs.
A discrete-time version of the model is obtained with the implicit Euler
discretization with sampling time $\tau = 0.1$. 
Both, the continuous- and discrete-time models have a single
unstable eigenvalue, $\neigr = 1$, which excites the unstable system dynamics.
To learn a basis of the unstable dynamics, we use $7$ gradient samples in the
discrete-time case.
For the continuous-time system, we use the approach described in
\Cref{subsubsec:krylov} with a shift in the right
open half-plane and GMRES for solving~\cref{eqn:linsolve} without a
preconditioner.
We need $192$ samples of the gradient to estimate the left eigenbasis.
Due to the homogeneous initial conditions, context-aware controller inference
(\Cref{alg:partstab}) needs $\dT = \neigr + 1 = 2$ samples to derive a
controller.

Four output measurements of the system are shown in \Cref{fig:hf2d5}.
Context-aware controller inference smoothly stabilize the dynamics in the
discrete- and continuous-time case.
As comparison, we train a controller via reinforcement
learning using DDPG for the discrete-time case.
The training performance is shown in
\Cref{fig:hf2d5_train_a}.
Reinforcement learning needs $20\,365$ state observations in $77$ episodes to
learn a controller that satisfies the stabilization test.
Note that the implemented test does not yield a guarantee for the constructed
controller to be stabilizing, which is in contrast to our proposed method that
yields a stabilization guarantee.
The number of state observations needed by DDPG is nearly $5$ times the
state-space dimension of the model of the system as well as $2\,260$ times
the number of observations needed by context-aware controller inference,
including the estimation of the eigenbasis.
This can be seen in \Cref{fig:hf2d5_train_b} that shows
the reward~\cref{eqn:reward} of context-aware controller.
The DDPG-learned controller also stabilizes the test simulation as shown
in \Cref{fig:hf2d5_e}.
The two inferred controllers provide very similar closed-loop simulation
behaviors in \Cref{fig:hf2d5_c,fig:hf2d5_d} due to the stabilization of single
real unstable eigenvalues while preserving the purely real eigenvalue structure
of the problem.
On the other hand, the DDPG-learned controller acts on the full spectrum
and introduces complex conjugate eigenvalues that dominate the dynamics leading
to an oscillatory behavior before converging to its long term steady state
behavior.
The DDPG controller performs also more aggressively than the inferred ones
resulting in overshooting of the trajectories when compared to the steady state.


\subsubsection{Brazilian interconnected power system}%
\label{subsubsec:bips}

\begin{figure}[t]
  \centering
  \begin{subfigure}[b]{.55\textwidth}
    \resizebox{\textwidth}{!}{%
  \tikzexternalenable%
  \tikzsetnextfilename{brazilianpowersystem}%
  \input{graphics/brazilianpowersystem.tikz}%
  \tikzexternaldisable%
}
    
    \caption{Sketch of the Brazilian interconnected power
      system~\cite{GENI, LosMRetal13}.}
    \label{fig:bips_dt_a}
  \end{subfigure}%
  \hfill%
  \begin{subfigure}[b]{.43\textwidth}
    \raggedleft

  \tikzexternalenable%
  \tikzsetnextfilename{bips_dt_sim_nofb}%
  \begin{tikzpicture}[font = \plotfontsize]
  \pgfplotstableread{graphics/data/bips_dt_sim_nofb.dat}\tableSIM
  
  \begin{axis}[%
    name   = states,
    width  = .725\textwidth,
    height = .1\textheight,
    scale only axis,
    xmin = 0,
    xmax = 100,
    ymin = -8,
    ymax = 45,
    xminorticks = false,
    yminorticks = false,
    xlabel = {time $\tau \cdot t$},
    ylabel = {output $y(t)$},
    ylabel style   = {yshift = -.3em},
    scaled x ticks = false,
    x tick label style = {/pgf/number format/1000 sep={\,}},
    y tick label style = {/pgf/number format/1000 sep={\,}},
    cycle list name    = stateslist
  ]
  
    \addplot+ table[x index = 0, y index = 5] {\tableSIM};
    \addplot+ table[x index = 0, y index = 6] {\tableSIM};
    \addplot+ table[x index = 0, y index = 9] {\tableSIM};
  \end{axis}
\end{tikzpicture}%
  \tikzexternaldisable%
\\
    \caption{DT: no controller.}
    \label{fig:bips_dt_b}
  \begin{subfigure}[b]{.43\textwidth}
  \end{subfigure}
    
    \raggedleft

  \tikzexternalenable%
  \tikzsetnextfilename{bips_dt_sim_inferfb}%
  \begin{tikzpicture}[font = \plotfontsize]
  \pgfplotstableread{graphics/data/bips_dt_sim_inferfb.dat}\tableSIM
  
  \begin{axis}[%
    name   = states,
    width  = .725\textwidth,
    height = .1\textheight,
    scale only axis,
    xmin = 0,
    xmax = 100,
    ymin = -20,
    ymax = 30,
    xminorticks = false,
    yminorticks = false,
    xlabel = {time $\tau \cdot t$},
    ylabel = {output $y(t)$},
    ylabel style   = {yshift = -.3em},
    scaled x ticks = false,
    x tick label style = {/pgf/number format/1000 sep={\,}},
    y tick label style = {/pgf/number format/1000 sep={\,}},
    cycle list name    = stateslist
  ]
  
    \addplot+ table[x index = 0, y index = 5] {\tableSIM};
    \addplot+ table[x index = 0, y index = 6] {\tableSIM};
    \addplot+ table[x index = 0, y index = 9] {\tableSIM};
  \end{axis}
\end{tikzpicture}%
  \tikzexternaldisable%
\\
    \caption{DT: context-aware controller.}
    \label{fig:bips_dt_c}
  \end{subfigure}
  
  \vspace{.5\baselineskip}
  \tikzexternalenable%
  \tikzsetnextfilename{bips_legend}%
  \begin{tikzpicture}[font = \plotfontsize]
  \begin{axis}[%
    hide axis,
    width  = .7\textwidth,
    height = .25\textwidth,
    scale only axis,
    xmin = 0,
    xmax = 10,
    ymin = 0.5,
    ymax = 1.5,
    legend columns = 5, 
    legend style = {
      at     = {(0,0)},
      anchor = center,
      /tikz/every even column/.append style = {column sep = 0.5cm}}
  ]
    
    \pgfplotsset{cycle list name = stateslist}
    \pgfplotsinvokeforeach{1, 2, 3}{\addplot coordinates {(0,0)};}
    \addlegendentry{output $y_{1}$}
    \addlegendentry{output $y_{2}$}
    \addlegendentry{output $y_{3}$}
  \end{axis}
\end{tikzpicture}%
  \tikzexternaldisable%

  \vspace{0\baselineskip}
  
  \caption{Power system:
    The growth of output $y_{3}$ indicates an instability of the system. 
    The proposed context-aware controller inference approach stabilizes the
    system leading to convergence of $y_{3}$ to 0.
    Only two state observations are necessary to infer a controller if a basis
    of the unstable dynamics is available.}
  \label{fig:bips_dt}
\end{figure}

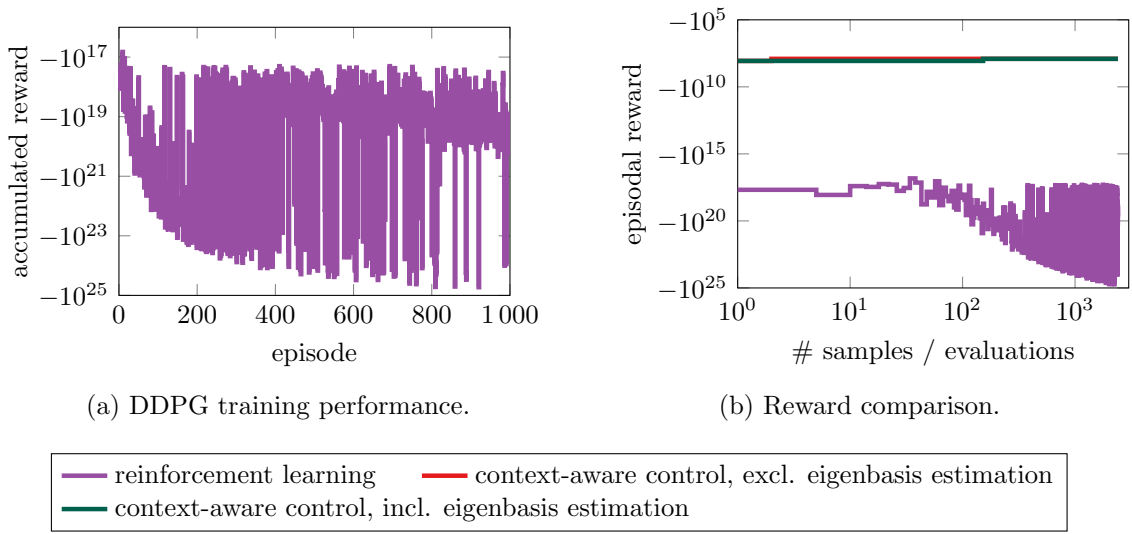
\begin{figure}[t]
  \centering%
  \begin{subfigure}[b]{.49\linewidth}
    \raggedleft
  \tikzexternalenable%
  \tikzsetnextfilename{bips_training}%
  \begin{tikzpicture}[font = \plotfontsize]
  \pgfplotstableread{graphics/data/bips_dt_learnfb_train.dat}\tableTRAIN
  \pgfplotstablecreatecol[create col/expr={-log10(-\thisrow{1})}]{}\tableTRAIN
  
  \begin{axis}[%
    width  = .7\textwidth,
    height = .15\textheight,
    scale only axis,
    xmin = 0,
    xmax = 1000,
    ymode = log,
    ymin = 1e+16,
    ymax = 1e+25,
    y dir       = reverse,
    ytick       = {1e+17, 1e+19, 1e+21, 1e+23, 1e+25},
    yticklabels = {$-10^{17}$, $-10^{19}$, $-10^{21}$, $-10^{23}$, $-10^{25}$},
    xminorticks = false,
    yminorticks = false,
    xlabel = {episode},
    ylabel = {accumulated reward},
    ylabel style   = {yshift = -.3em},
    scaled x ticks = false,
    x tick label style = {/pgf/number format/1000 sep={\,}},
    y tick label style = {/pgf/number format/1000 sep={\,}},
    cycle list name    = trainlist,
    cycle list shift   = 2
  ]
  
    \addplot+[const plot] table[x index = 0, y expr = -\thisrowno{1}]
      {\tableTRAIN};
  \end{axis}
\end{tikzpicture}%
  \tikzexternaldisable%

    \caption{DDPG training performance.}
    \label{fig:bips_train_a}
  \end{subfigure}%
  \hfill%
  \begin{subfigure}[b]{.49\linewidth}
    \raggedleft
  \tikzexternalenable%
  \tikzsetnextfilename{bips_rewards}%
  \begin{tikzpicture}[font = \plotfontsize]
  \pgfplotstableread{graphics/data/bips_dt_rewards.dat}\tableREW
  
  \begin{loglogaxis}[%
    width  = .7\textwidth,
    height = .15\textheight,
    scale only axis,
    xmin = 1,
    xmax = 3e+3,
    ymin = 1e+5,
    ymax = 1e+25,
    y dir       = reverse,
    ytick       = {1e+5, 1e+10, 1e+15, 1e+20, 1e+25},
    yticklabels = {$-10^{5}$, $-10^{10}$, $-10^{15}$, $-10^{20}$, $-10^{25}$},
    xminorticks = false,
    yminorticks = false,
    xlabel = {\# samples / evaluations},
    ylabel = {episodal reward},
    ylabel style   = {yshift = -.3em},
    scaled x ticks = false,
    x tick label style = {/pgf/number format/1000 sep={\,}},
    y tick label style = {/pgf/number format/1000 sep={\,}},
    cycle list name    = trainlist
  ]
  
    \addplot+[const plot] table[x index = 0, y expr = -\thisrowno{1}] {\tableREW};
    \addplot+[const plot] table[x index = 0, y expr = -\thisrowno{2}] {\tableREW};
    \addplot+[const plot] table[x index = 0, y expr = -\thisrowno{3}] {\tableREW};
  \end{loglogaxis}
\end{tikzpicture}%
  \tikzexternaldisable%

    \caption{Reward comparison.}
    \label{fig:bips_train_b}
  \end{subfigure}
  
  \vspace{.75\baselineskip}
  \tikzexternalenable%
  \tikzsetnextfilename{bips_train_legend}%
  \begin{tikzpicture}[font = \plotfontsize]
  \begin{axis}[%
    hide axis,
    width  = 1cm,
    height = 1cm,
    scale only axis,
    xmin = 0,
    xmax = 10,
    ymin = 0.5,
    ymax = 1.5,
    legend columns    = 2,
    legend cell align = {left},
    legend style      = {
      at     = {(0,0)},
      anchor = center,
      /tikz/every even column/.append style = {column sep = 0.5cm}}
  ]
    
    \pgfplotsset{cycle list name = trainlist, cycle list shift = 2}
    \pgfplotsinvokeforeach{1, 2, 3}{\addplot coordinates {(0,0)};}
    \addlegendentry{reinforcement learning}
    \addlegendentry{context-aware control, excl. eigenbasis estimation}
    \addlegendentry{\makebox[0pt][l]{context-aware control, incl.
      eigenbasis estimation}}
  \end{axis}
\end{tikzpicture}%
  \tikzexternaldisable%

  \vspace{0\baselineskip}
  
  \caption{Power system:
    Reinforcement learning with DDPG runs for $1\,000$ episodes but in each
    episode at most five time steps can be taken before the system dynamics
    become too unstable and lead to numerical issues, which ultimately means
    that no stabilizing controller is found with DDPG.
    In contrast, context-aware controller inference stabilizes the system after
    only two state observations if an estimate of a basis of the unstable
    dynamics is given and $152$ samples if additionally the basis has to
    be estimated.}
  \label{fig:bips_train}
\end{figure}

We now consider a system that describes the Brazilian interconnected power
system under heavy load condition.
A sketch is shown in \Cref{fig:bips_dt_a}.
The network consists of generators and power plants, consumer clusters and
industrial loads, and transmission lines, and its internal energy is controlled
by reference voltage excitation; see, e.g.,~\cite{LosMRetal13,FreRM08}.
Instabilities in the system occur by deviations from the considered heavy load
condition, e.g., by disturbances, and result in an increase of the internal
network energy, leading to shorts and blackouts.
We consider here the \textsf{bips98\_606} example
from~\cite{morwiki_powersystems}.
The model has $\nh = 7\,135$ states, described by ordinary differential
equations and algebraic constraints, and $\np = 4$ inputs.
Due to the algebraic constraints, we only consider the
discrete-time case, which is obtained with implicit Euler and sampling time
$\tau = 0.1$.
The model has $\neigr = 1$ unstable eigenvalues, for which we learn the
corresponding left eigenvector as basis of the unstable dynamics from
$150$ gradient samples.

Trajectories of three output measurements up to time $100$ are shown in
\Cref{fig:bips_dt_b}.
The first two outputs represent the voltage at two locations in the network.
Those two outputs do not yet indicate an instability.
However, the average of the total network energy, which is given by the third
output, increases rapidly over time in \Cref{fig:bips_dt_b}
and so indicates the accumulation of internal energy.
From only $\dT = 2$ data samples of the system, we construct a stabilizing
controller via context-aware inference.
Outputs of the system with the inferred controller are shown in
\Cref{fig:bips_dt_c} and demonstrate that the controlled
system dampens the internal energy (output $y_{3}$).

We also attempted to stabilize the system with DDPG.
However, even after manual parameter tuning and $1\,000$ episodes, we only
obtained controllers that destabilized the system.
The destabilization with DDPG happened so quickly that typically only
$2$--$5$ time steps per episode were sampled, which shows that sampling data of
unstable systems is challenging because after only a short time the instability
leads to uninformative data collection.
The training performance of DDPG is compared to controller inference in
\Cref{fig:bips_train}.


\subsection{Experiments with nonlinear system dynamics}%
\label{subsec:nlexamples}


\subsubsection{Tubular reactor}%
\label{subsubsec:tubularreactor}

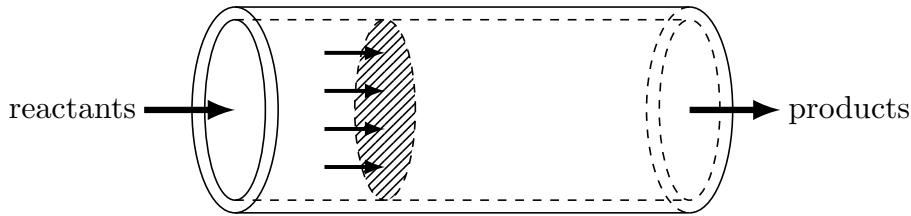
\begin{figure}[t]
  \centering
  \resizebox{.8\textwidth}{!}{%
  \tikzexternalenable%
  \tikzsetnextfilename{tubularreactor}%
%
%
%

\begin{tikzpicture}[
  x         = 2.845em,
  y         = 2.845em,
  inner sep = 0,
  outer sep = 0,
  draw      = black]
  
  \tikzstyle{outerwall} = [
    line width = .05em
  ]
  \tikzstyle{innerwall} = [
    outerwall,
    dashed
  ]
  \makeatletter
  \tikzset{
    hatch distance/.store in = \hatchdistance,
    hatch distance = .3em,
    hatch thickness/.store in = \hatchthickness,
    hatch thickness = 0.08
  }
  \pgfdeclarepatternformonly[\hatchdistance,\hatchthickness]{north east hatch}
    {\pgfpoint{-3em}{-3em}}
    {\pgfpoint{\hatchdistance+.5*\hatchthickness}%
	    {\hatchdistance+.5*\hatchthickness}}
    {\pgfpoint{\hatchdistance}{\hatchdistance}}
    {\pgfsetcolor{\tikz@pattern@color}
    	\pgfsetlinewidth{\hatchthickness}
    	\pgfpathmoveto{\pgfpoint{0em}{0em}}
    	\pgfpathlineto{\pgfpoint{\hatchdistance}{\hatchdistance}}
    	\pgfusepath{stroke}}
  \makeatother
  
  \node(react) {reactants\vphantom{Pp}};
  \draw[-latex, line width = .2em]
    ($(react.east)+(.25em,0)$) -- ++(1, 0)
    coordinate (lhs);
  \coordinate (rhs) at ($(lhs)+(5,0)$);
  
  \draw[outerwall] (lhs) ellipse (0.9483em and 2.845em);
  \draw[outerwall] (lhs) ellipse (1.3483em and 3.245em);
  
  \draw[outerwall] ($(lhs)+(0,-3.245em)$) -- ++(5, 0);
  \draw[outerwall] ($(lhs)+(0,3.245em)$) -- ++(5, 0);
  
  \draw[innerwall] ($(lhs)+(0,-2.845em)$) -- ++(5, 0);
  \draw[innerwall] ($(lhs)+(0,2.845em)$) -- ++(5, 0);
  
  \draw[innerwall] (rhs) ellipse (0.9483em and 2.845em);
  \draw[innerwall] ($(lhs)+(5, 3.245em)$)
    arc (90:270:1.3483em and 3.245em);
  \draw[outerwall] ($(lhs)+(5, 3.245em)$)
    arc (90:-90:1.3483em and 3.245em);
    
  \draw[innerwall, pattern = north east hatch, hatch distance = .28em, hatch thickness = .05em]
    ($(lhs)+(1.65,0)$)
    coordinate (mdl)
    ellipse (0.9483em and 2.845em);
    
  \draw[-latex, line width = .125em] ($(mdl)+(-1.8966em,.6em)$)
    -- ($(mdl)+(0,.6em)$);
  \draw[-latex, line width = .125em] ($(mdl)+(-1.8966em,1.8em)$)
    -- ($(mdl)+(0,1.8em)$);
  \draw[-latex, line width = .125em] ($(mdl)+(-1.8966em,-.6em)$)
    -- ($(mdl)+(0,-.6em)$);
  \draw[-latex, line width = .125em] ($(mdl)+(-1.8966em,-1.8em)$)
    -- ($(mdl)+(0,-1.8em)$);
    
  \draw[-latex, line width = .2em]
    (rhs) -- ++(1, 0)
    coordinate (tmp);
  \node(prod) [right = .25em of tmp] {products\vphantom{Pp}};
\end{tikzpicture} %
  \tikzexternaldisable%
}
  
  \caption{Tubular reactor: Schematic representation of a tubular reactor.
    Reactants are injected, react inside the reactor, and the products are
    returned.
    The control goal is dampening temperature oscillations during the reaction.}
  \label{fig:tubularreactorsketch}
\end{figure}

\begin{figure}
  \centering%
  \begin{subfigure}[b]{.49\linewidth}
    \centering
  \tikzexternalenable%
  \tikzsetnextfilename{tubularreactor_ct_sim_nofb}%
  \begin{tikzpicture}[font = \plotfontsize]  
  \begin{axis}[%
    name   = states,
    width  = .725\textwidth,
    height = .135\textheight,
    scale only axis,
    xmin = 0,
    xmax = 30,
    ymin = 0,
    ymax = 1,
    xminorticks = false,
    yminorticks = false,
    xlabel = {time $t$},
    ylabel = {spatial coordinates},
    ylabel style   = {yshift = -.3em},
    scaled x ticks = false,
    x tick label style = {/pgf/number format/1000 sep={\,}},
    y tick label style = {/pgf/number format/1000 sep={\,}}
  ]
  
    \addplot graphics[xmin = 0, xmax = 30, ymin = 0, ymax = 1]
        {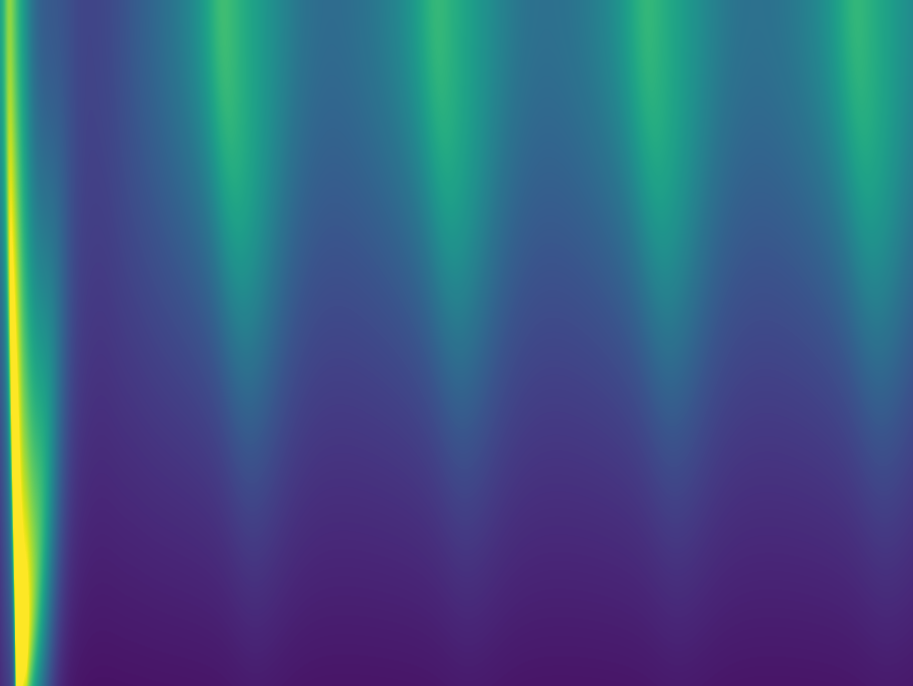};
  \end{axis}
\end{tikzpicture}%
  \tikzexternaldisable%

    \caption{CT: no controller.}
    \label{fig:tubularreactor_a}
  \end{subfigure}%
  \hfill%
  \begin{subfigure}[b]{.49\linewidth}
    \centering
  \tikzexternalenable%
  \tikzsetnextfilename{tubularreactor_dt_sim_nofb}%
  \begin{tikzpicture}[font = \plotfontsize]  
  \begin{axis}[%
    name   = states,
    width  = .725\textwidth,
    height = .135\textheight,
    scale only axis,
    xmin = 0,
    xmax = 30,
    ymin = 0,
    ymax = 1,
    xminorticks = false,
    yminorticks = false,
    xlabel = {time $\tau \cdot t$},
    ylabel = {spatial coordinates},
    ylabel style   = {yshift = -.3em},
    scaled x ticks = false,
    x tick label style = {/pgf/number format/1000 sep={\,}},
    y tick label style = {/pgf/number format/1000 sep={\,}}
  ]
  
    \addplot graphics[xmin = 0, xmax = 30, ymin = 0, ymax = 1]
        {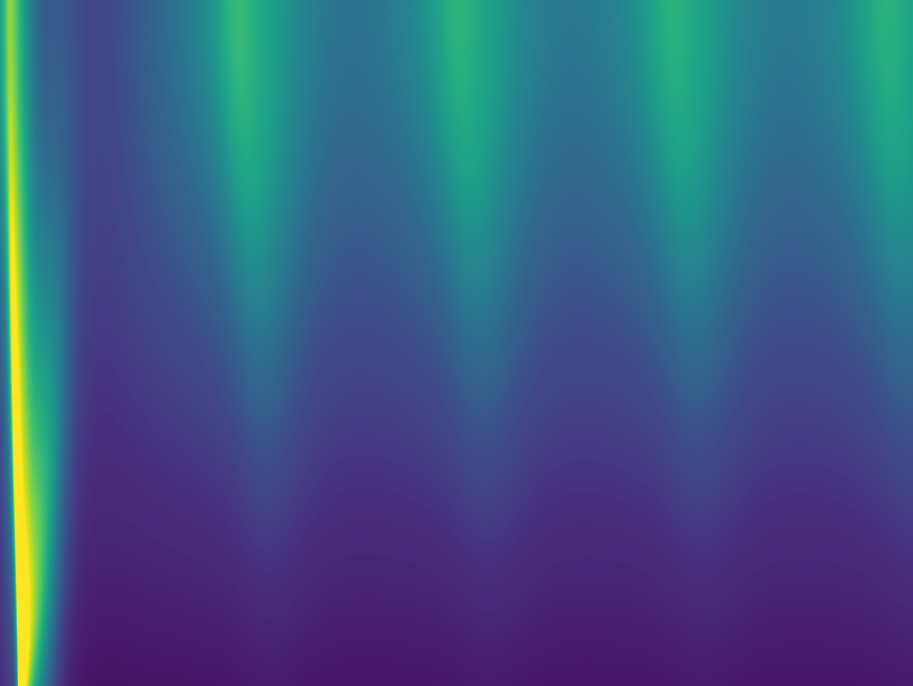};
  \end{axis}
\end{tikzpicture}%
  \tikzexternaldisable%

    \caption{DT: no controller.}
    \label{fig:tubularreactor_b}
  \end{subfigure}
  \vspace{1\baselineskip}
  
  \begin{subfigure}[b]{.49\linewidth}
    \centering
  \tikzexternalenable%
  \tikzsetnextfilename{tubularreactor_ct_sim_inferfb}%
  \begin{tikzpicture}[font = \plotfontsize]  
  \begin{axis}[%
    name   = states,
    width  = .725\textwidth,
    height = .135\textheight,
    scale only axis,
    xmin = 0,
    xmax = 30,
    ymin = 0,
    ymax = 1,
    xminorticks = false,
    yminorticks = false,
    xlabel = {time $t$},
    ylabel = {spatial coordinates},
    ylabel style   = {yshift = -.3em},
    scaled x ticks = false,
    x tick label style = {/pgf/number format/1000 sep={\,}},
    y tick label style = {/pgf/number format/1000 sep={\,}}
  ]
  
    \addplot graphics[xmin = 0, xmax = 30, ymin = 0, ymax = 1]
        {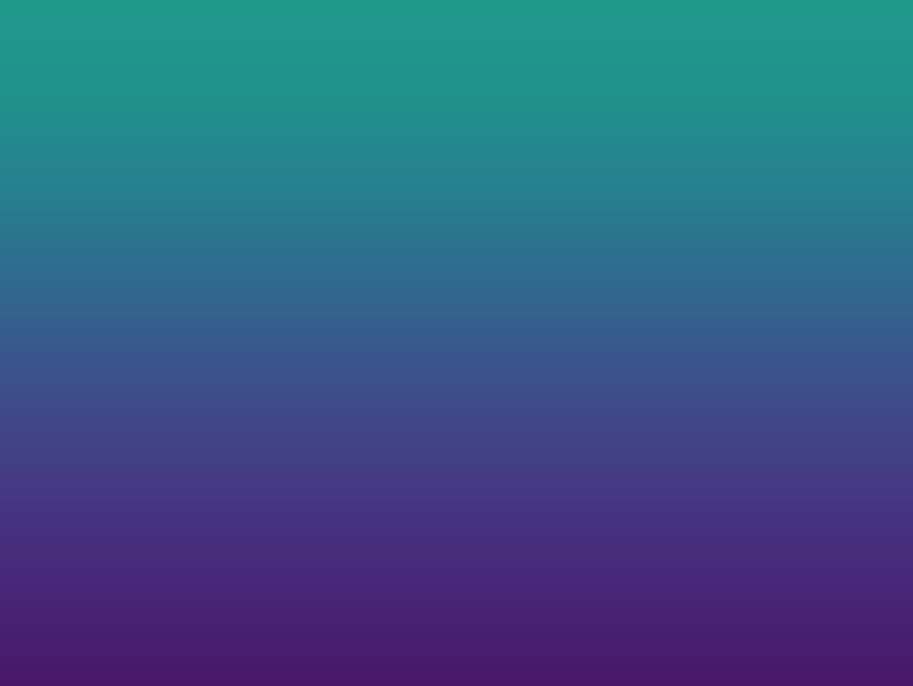};
  \end{axis}
\end{tikzpicture}%
  \tikzexternaldisable%

    \caption{CT: context aware, data from linearized system.}
    \label{fig:tubularreactor_c}
  \end{subfigure}%
  \hfill%
  \begin{subfigure}[b]{.49\linewidth}
    \centering
  \tikzexternalenable%
  \tikzsetnextfilename{tubularreactor_dt_sim_inferfb}%
  \begin{tikzpicture}[font = \plotfontsize]  
  \begin{axis}[%
    name   = states,
    width  = .725\textwidth,
    height = .135\textheight,
    scale only axis,
    xmin = 0,
    xmax = 30,
    ymin = 0,
    ymax = 1,
    xminorticks = false,
    yminorticks = false,
    xlabel = {time $\tau \cdot t$},
    ylabel = {spatial coordinates},
    ylabel style   = {yshift = -.3em},
    scaled x ticks = false,
    x tick label style = {/pgf/number format/1000 sep={\,}},
    y tick label style = {/pgf/number format/1000 sep={\,}}
  ]
  
    \addplot graphics[xmin = 0, xmax = 30, ymin = 0, ymax = 1]
        {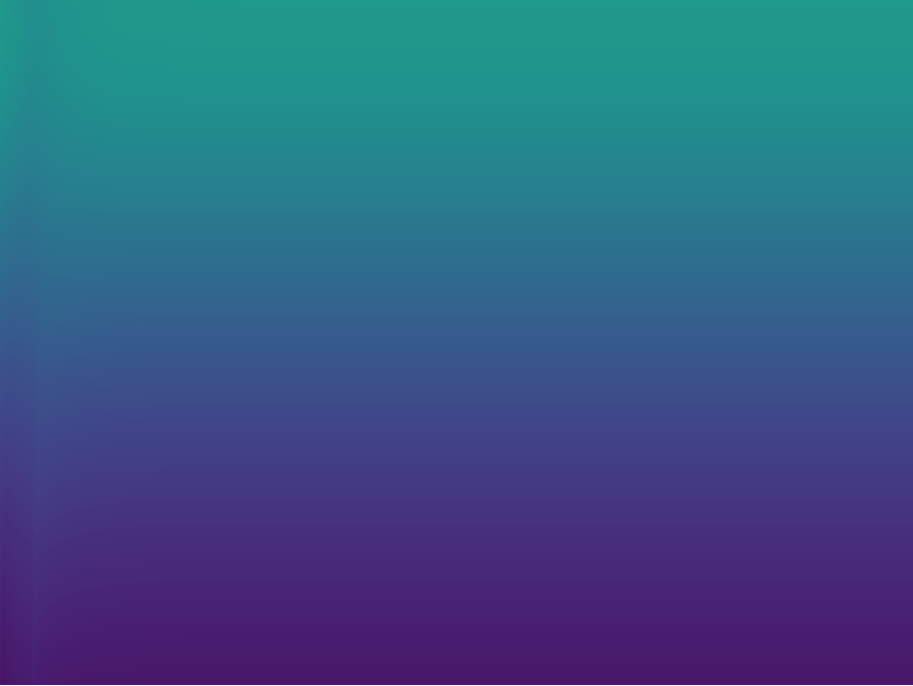};
  \end{axis}
\end{tikzpicture}%
  \tikzexternaldisable%

    \caption{DT: context aware, data from linearized system.}
    \label{fig:tubularreactor_d}
  \end{subfigure}
  \vspace{1\baselineskip}
  
  \begin{subfigure}[b]{.49\linewidth}
    \centering
  \tikzexternalenable%
  \tikzsetnextfilename{tubularreactor_ct_sim_inferfb_nl}%
  \begin{tikzpicture}[font = \plotfontsize]  
  \begin{axis}[%
    name   = states,
    width  = .725\textwidth,
    height = .135\textheight,
    scale only axis,
    xmin = 0,
    xmax = 30,
    ymin = 0,
    ymax = 1,
    xminorticks = false,
    yminorticks = false,
    xlabel = {time $t$},
    ylabel = {spatial coordinates},
    ylabel style   = {yshift = -.3em},
    scaled x ticks = false,
    x tick label style = {/pgf/number format/1000 sep={\,}},
    y tick label style = {/pgf/number format/1000 sep={\,}}
  ]
  
    \addplot graphics[xmin = 0, xmax = 30, ymin = 0, ymax = 1]
        {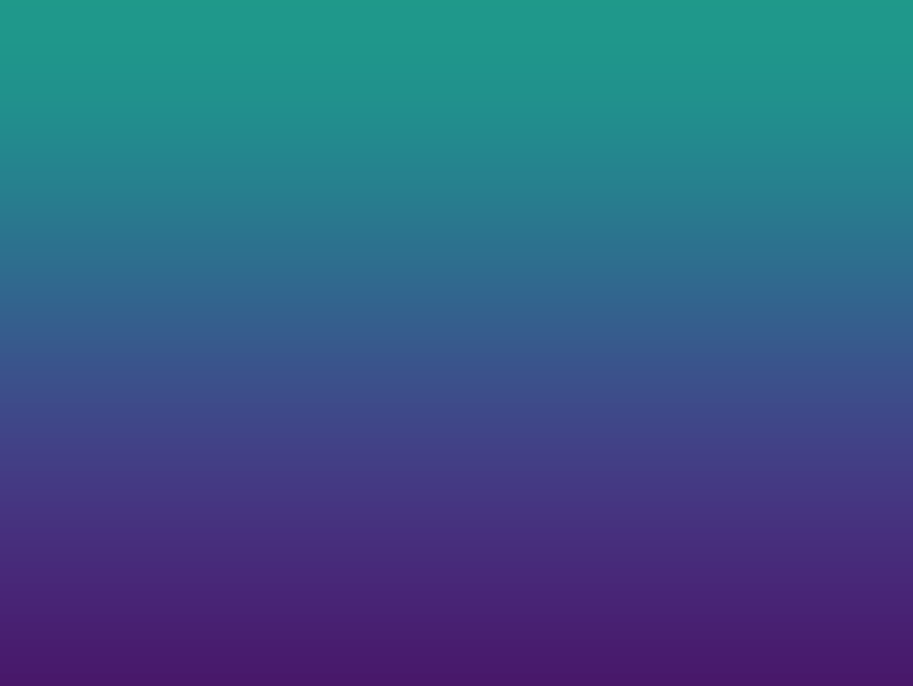};
  \end{axis}
\end{tikzpicture}%
  \tikzexternaldisable%

    \caption{CT: context aware, data from nonlinear system.}
    \label{fig:tubularreactor_e}
  \end{subfigure}%
  \hfill%
  \begin{subfigure}[b]{.49\linewidth}
    \centering
  \tikzexternalenable%
  \tikzsetnextfilename{tubularreactor_dt_sim_inferfb_nl}%
  \begin{tikzpicture}[font = \plotfontsize]  
  \begin{axis}[%
    name   = states,
    width  = .725\textwidth,
    height = .135\textheight,
    scale only axis,
    xmin = 0,
    xmax = 30,
    ymin = 0,
    ymax = 1,
    xminorticks = false,
    yminorticks = false,
    xlabel = {time $\tau \cdot t$},
    ylabel = {spatial coordinates},
    ylabel style   = {yshift = -.3em},
    scaled x ticks = false,
    x tick label style = {/pgf/number format/1000 sep={\,}},
    y tick label style = {/pgf/number format/1000 sep={\,}}
  ]
  
    \addplot graphics[xmin = 0, xmax = 30, ymin = 0, ymax = 1]
        {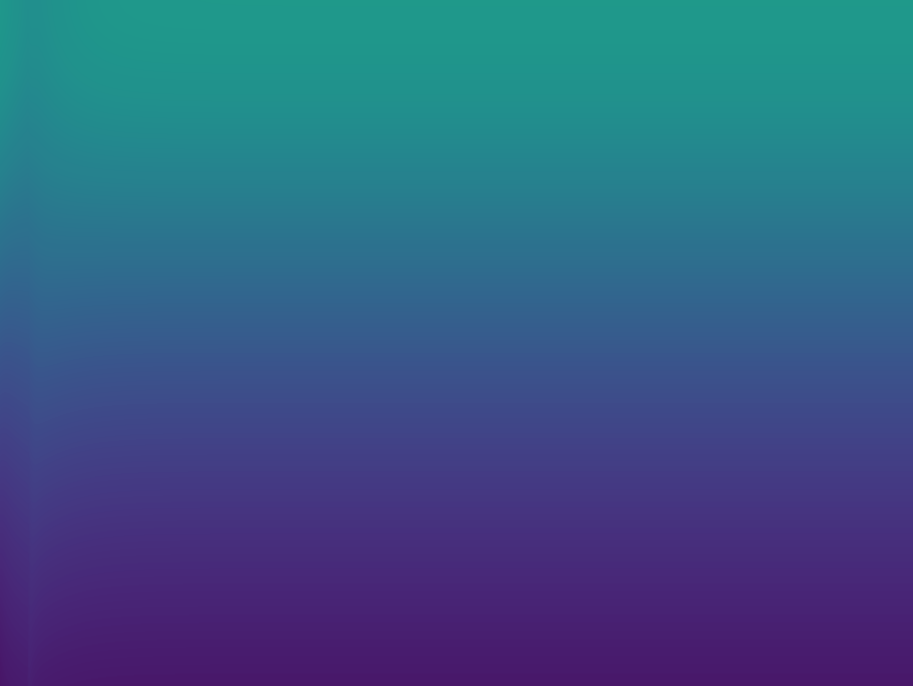};
  \end{axis}
\end{tikzpicture}%
  \tikzexternaldisable%

    \caption{DT: context aware, data from nonlinear system.}
    \label{fig:tubularreactor_f}
  \end{subfigure}

  \vspace{.75\baselineskip}
  \tikzexternalenable%
  \tikzsetnextfilename{tubularreactor_legend}%
  \begin{tikzpicture}
  \node[draw = none, minimum width = 0cm, inner sep = 0cm](start){};
  \node(leg) at (start.north east) [anchor = north west]{\tikz
  \begin{axis}[%
    hide axis,
    scale only axis,
    width  = 10cm,
    height = .1cm,
    point meta min = 1.0,
    point meta max = 1.3,
    colorbar,
    colorbar horizontal,
    colorbar style = {
      at = {(.5, 0)},
      anchor = north},
    scaled x ticks = false,
    x tick label style = {/pgf/number format/fixed}]
  \end{axis};};
  \node[draw = none, minimum width = 0cm, inner sep = 0cm](end)
    at (leg.north east) [anchor = north west]{};
\end{tikzpicture}%
  \tikzexternaldisable%

  \vspace{0\baselineskip}

  \caption{Tubular reactor:
    The proposed context-aware controller inference stabilizes the system
    already after observing one order of magnitude fewer states than necessary
    for stabilizing via the traditional two-step approach.
    Plots (e) and (f) indicate that the proposed approach also works well on
    data directly obtained from the nonlinear systems.}
  \label{fig:tubularreactor}
\end{figure}

\begin{figure}[t]
  \centering%
  \begin{subfigure}[b]{.49\linewidth}
    \raggedleft
  \tikzexternalenable%
  \tikzsetnextfilename{tubularreactor_training}%
  \begin{tikzpicture}[font = \plotfontsize]
  \pgfplotstableread{graphics/data/tubularreactor_dt_learnfb_train.dat}\tableTRAIN
  \pgfplotstablecreatecol[create col/expr={-log10(-\thisrow{1})}]{}\tableTRAIN
  
  \begin{axis}[%
    width  = .7\textwidth,
    height = .15\textheight,
    scale only axis,
    xmin = 0,
    xmax = 1000,
    ymode = log,
    ymin = 1e+10,
    ymax = 1e+19,
    y dir       = reverse,
    ytick       = {1e+10, 1e+13, 1e+16, 1e+19},
    yticklabels = {$-10^{10}$, $-10^{13}$, $-10^{16}$, $-10^{19}$},
    xminorticks = false,
    yminorticks = false,
    xlabel = {episode},
    ylabel = {accumulated reward},
    ylabel style   = {yshift = -.3em},
    scaled x ticks = false,
    x tick label style = {/pgf/number format/1000 sep={\,}},
    y tick label style = {/pgf/number format/1000 sep={\,}},
    cycle list name    = trainlist,
    cycle list shift   = 2
  ]
  
    \addplot+[const plot] table[x index = 0, y expr = -\thisrowno{1}]
      {\tableTRAIN};
  \end{axis}
\end{tikzpicture}%
  \tikzexternaldisable%

    \caption{DDPG training performance.}
    \label{fig:tubularreactor_train_a}
  \end{subfigure}%
  \hfill%
  \begin{subfigure}[b]{.49\linewidth}
    \raggedleft
  \tikzexternalenable%
  \tikzsetnextfilename{tubularreactor_rewards}%
  \begin{tikzpicture}[font = \plotfontsize]
  \pgfplotstableread{graphics/data/tubularreactor_dt_rewards.dat}\tableREW
  
  \begin{loglogaxis}[%
    width  = .7\textwidth,
    height = .15\textheight,
    scale only axis,
    xmin = 1,
    xmax = 4e+3,
    ymin = 1e-1,
    ymax = 1e+20,
    y dir       = reverse,
    ytick       = {1e+0, 1e+5, 1e+10, 1e+15, 1e+20},
    yticklabels = {$-10^{0}$, $-10^{5}$, $-10^{10}$, $-10^{15}$, $-10^{20}$},
    xminorticks = false,
    yminorticks = false,
    xlabel = {\# samples / evaluations},
    ylabel = {episodal reward},
    ylabel style   = {yshift = -.3em},
    scaled x ticks = false,
    x tick label style = {/pgf/number format/1000 sep={\,}},
    y tick label style = {/pgf/number format/1000 sep={\,}},
    cycle list name    = trainlist
  ]
  
    \addplot+[const plot] table[x index = 0, y expr = -\thisrowno{1}] {\tableREW};
    \addplot+[const plot] table[x index = 0, y expr = -\thisrowno{2}] {\tableREW};
    \addplot+[const plot] table[x index = 0, y expr = -\thisrowno{3}] {\tableREW};
  \end{loglogaxis}
\end{tikzpicture}%
  \tikzexternaldisable%

    \caption{Reward comparison.}
    \label{fig:tubularreactor_train_b}
  \end{subfigure}
  
  \vspace{.75\baselineskip}
  \tikzexternalenable%
  \tikzsetnextfilename{tubularreactor_train_legend}%
  \begin{tikzpicture}[font = \plotfontsize]
  \begin{axis}[%
    hide axis,
    width  = 1cm,
    height = 1cm,
    scale only axis,
    xmin = 0,
    xmax = 10,
    ymin = 0.5,
    ymax = 1.5,
    legend columns    = 2,
    legend cell align = {left},
    legend style      = {
      at     = {(0,0)},
      anchor = center,
      /tikz/every even column/.append style = {column sep = 0.5cm}}
  ]
    
    \pgfplotsset{cycle list name = trainlist, cycle list shift = 2}
    \pgfplotsinvokeforeach{1, 2, 3}{\addplot coordinates {(0,0)};}
    \addlegendentry{reinforcement learning}
    \addlegendentry{context-aware control, excl. eigenbasis estimation}
    \addlegendentry{\makebox[0pt][l]{context-aware control, incl.
      eigenbasis estimation}}
  \end{axis}
\end{tikzpicture}%
  \tikzexternaldisable%

  \vspace{0\baselineskip}
  
  \caption{Tubular reactor:
    Reinforcement learning with DDPG fails to find
    a stabilizing controller after $1\,000$ episodes in this example.
    The context-aware approach needs around $260\times$ fewer samples than what
    DDPG uses in $1\,000$  episodes.}
  \label{fig:tubularreactor_train}
\end{figure}
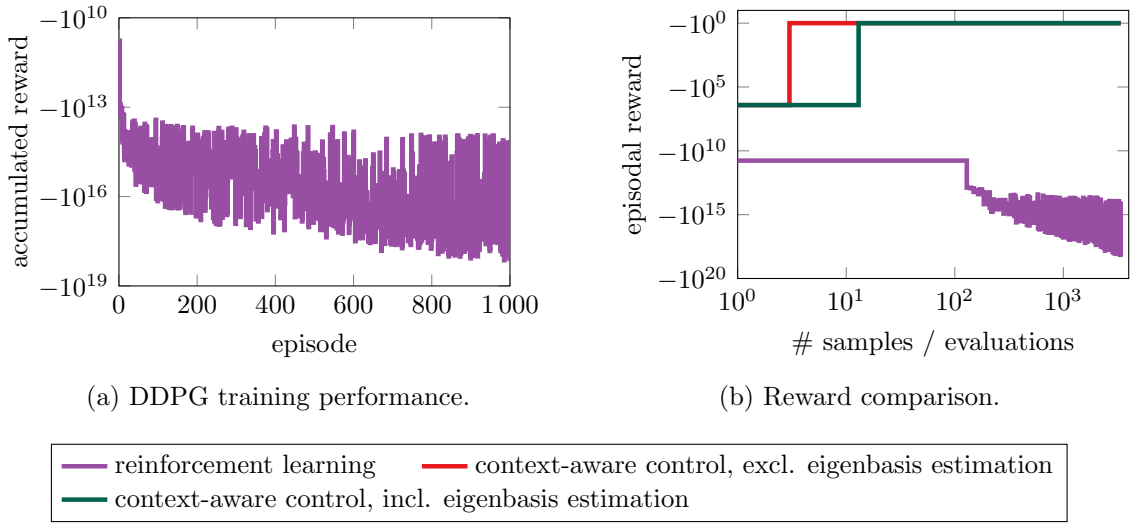

We consider a tubular reactor as shown in
\Cref{fig:tubularreactorsketch} and described in~\cite{KraW19, Zho12}.
The reaction is given by an Arrhenius term that relates temperature and
specimen concentration to describe a single exothermic reaction~\cite{HeiP81}.
The nonlinear model is a one-dimensional discretization of the reaction
equations for \Cref{fig:tubularreactorsketch} with $\nh = 3\,998$ states and
$\np = 2$ control inputs.
The inputs influence the temperature and specimen concentration at
the left end of the tube.
Additionally to the continuous-time model, we consider a discrete-time version
obtained using implicit Euler for the linear part and explicit Euler for the
nonlinear part with the sampling time $\tau = 0.01$.
We note that our approach is applicable with other time-discretization
schemes as well.

Simulations without a controller showing the unstable behavior for the first
half of the states representing temperature are given in
\Cref{fig:tubularreactor_a,fig:tubularreactor_b}.
The models have $\neigr = 2$ unstable eigenvalues, for which
we need $41$ and $10$ gradient samples to estimate a basis of the
unstable dynamics in the continuous and discrete time case, respectively.
The observations in the continuous-time case include the use of GMRES
to solve the occurring shifted systems of linear equations during the eigenvalue
computations.
Inspired by the reaction-diffusion nature of the problem, we use a shifted
matrix resulting from the discretization of a linear diffusion equation as
preconditioner.
In many cases, with knowledge about the underlying application, preconditioners
can be designed without system identification~\cite{PeaP20}.

With the learned bases of the unstable dynamics, we use in continuous and
discrete time $\dT = \neigr + 1 = 3$ data samples to design stabilizing
controllers.
\Cref{fig:tubularreactor_c,fig:tubularreactor_d} show the simulations
with controllers inferred from data of the models of the linearized systems
(``idealized data'').
In contrast, \Cref{fig:tubularreactor_e,fig:tubularreactor_f} show the results
obtained with controllers inferred from data of the original, nonlinear systems,
which are sampled close to the steady state of interest.

We also trained a DDPG agent to stabilize the discrete-time system based on
evaluations of the corresponding linear models.
Even after $1\,000$ episodes, the training did not result in a controller that
satisfied the stabilization criterion.
The training results are shown in \Cref{fig:tubularreactor_train_a}.
In fact, the constructed DDPG controllers often destabilize the system
in the sense of~\cref{eqn:reward} after 3--4 time steps.
In \Cref{fig:tubularreactor_train_b},
we compare the two types of constructed controllers in terms of their
rewards per data samples.
The controller constructed with the proposed method stabilizes the system
with fewer samples than what DDPG requires in the first training episode alone.


\subsubsection{Laminar flow with an obstacle}%
\label{subsubsec:cylinderwake}

\begin{figure}
  \centering
  \begin{subfigure}[b]{\linewidth}
    \raggedleft
  \tikzexternalenable%
  \tikzsetnextfilename{cylinderwake_dt_sim_nofb}%
  \input{graphics/cylinderwake_dt_sim_nofb.tikz}%
  \tikzexternaldisable%

    \tikz\node[minimum size = 1.5cm]{};
    
    \caption{DT: no controller.}
    \label{fig:cylinderwake_dt_a}
  \end{subfigure}
  \vspace{0\baselineskip}

  \begin{subfigure}[b]{\linewidth}
    \raggedleft
  \tikzexternalenable%
  \tikzsetnextfilename{cylinderwake_dt_sim_inferfb}%
  \input{graphics/cylinderwake_dt_sim_inferfb.tikz}%
  \tikzexternaldisable%

    \tikz\node[minimum size = 1.5cm]{};
    
    \caption{DT: context-aware controller obtained from data of the
      linearized system.}
    \label{fig:cylinderwake_dt_b}
  \end{subfigure}
  \vspace{0\baselineskip}

  \begin{subfigure}[b]{\linewidth}
    \raggedleft
  \tikzexternalenable%
  \tikzsetnextfilename{cylinderwake_dt_sim_inferfb_nl}%
  \input{graphics/cylinderwake_dt_sim_inferfb_nl.tikz}%
  \tikzexternaldisable%

    \tikz\node[minimum size = 1.5cm]{};
    
    \caption{DT: context-aware controller obtained from data of the
      nonlinear system.}
    \label{fig:cylinderwake_dt_c}
  \end{subfigure}
  
  \vspace{.5\baselineskip}
  \tikzexternalenable%
  \tikzsetnextfilename{cylinderwake_legend}%
  \begin{tikzpicture}[font = \plotfontsize]
  \begin{axis}[%
    hide axis,
    width  = .75\textwidth,
    height = .1\textheight,
    scale only axis,
    xmin = 0,
    xmax = 10,
    ymin = 0.5,
    ymax = 1.5,
    legend columns = 5, 
    legend style = {
      at     = {(0,0)},
      anchor = center,
      /tikz/every even column/.append style = {column sep = 0.5cm}}
  ]
    
    \pgfplotsset{cycle list name = stateslist}
    \pgfplotsinvokeforeach{1, 2, 3, 4}{\addplot coordinates {(0,0)};}
    \addlegendentry{output $y_{1}$}
    \addlegendentry{output $y_{3}$}
    \addlegendentry{output $y_{5}$}
    \addlegendentry{output $y_{7}$}
  \end{axis}
\end{tikzpicture}%
  \tikzexternaldisable%

  \vspace{0\baselineskip}
  
  \caption{Flow behind a cylinder:
    The proposed context-aware controller inference stabilizes the system with
    state observations obtained from the linearized and nonlinear systems.
    The outputs near the disturbance initiating the unstable behavior are less
    oscillatory in the case of using data from the nonlinear system in this
    experiment.}
  \label{fig:cylinderwake_dt}
\end{figure}

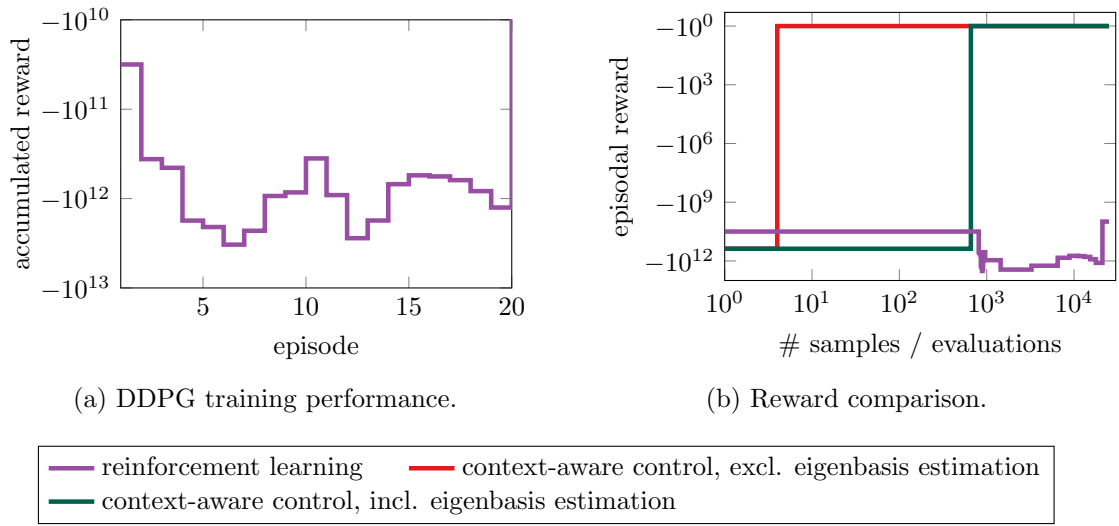
\begin{figure}[t]
  \centering%
  \begin{subfigure}[b]{.49\linewidth}
    \raggedleft
  \tikzexternalenable%
  \tikzsetnextfilename{cylinderwake_training}%
  \begin{tikzpicture}[font = \plotfontsize]
  \pgfplotstableread{graphics/data/cylinderwake_dt_learnfb_train.dat}\tableTRAIN
  \pgfplotstablecreatecol[create col/expr={-log10(-\thisrow{1})}]{}\tableTRAIN
  
  \begin{axis}[%
    width  = .7\textwidth,
    height = .15\textheight,
    scale only axis,
    xmin = 1,
    xmax = 20,
    ymode = log,
    ymin = 1e+10,
    ymax = 1e+13,
    y dir       = reverse,
    ytick       = {1e+10, 1e+11, 1e+12, 1e+13},
    yticklabels = {$-10^{10}$, $-10^{11}$, $-10^{12}$, $-10^{13}$},
    xminorticks = false,
    yminorticks = false,
    xlabel = {episode},
    ylabel = {accumulated reward},
    ylabel style   = {yshift = -.3em},
    scaled x ticks = false,
    x tick label style = {/pgf/number format/1000 sep={\,}},
    y tick label style = {/pgf/number format/1000 sep={\,}},
    cycle list name    = trainlist,
    cycle list shift   = 2
  ]
  
    \addplot+[const plot] table[x index = 0, y expr = -\thisrowno{1}]
      {\tableTRAIN};
  \end{axis}
\end{tikzpicture}%
  \tikzexternaldisable%

    \caption{DDPG training performance.}
    \label{fig:cylinderwake_train_a}
  \end{subfigure}%
  \hfill%
  \begin{subfigure}[b]{.49\linewidth}
    \raggedleft
  \tikzexternalenable%
  \tikzsetnextfilename{cylinderwake_rewards}%
  \begin{tikzpicture}[font = \plotfontsize]
  \pgfplotstableread{graphics/data/cylinderwake_dt_rewards.dat}\tableREW
  
  \begin{loglogaxis}[%
    width  = .7\textwidth,
    height = .15\textheight,
    scale only axis,
    xmin = 1,
    xmax = 3e+4,
    ymin = 2e-1,
    ymax = 1e+13,
    y dir       = reverse,
    ytick       = {1e+0, 1e+3, 1e+6, 1e+9, 1e+12},
    yticklabels = {$-10^{0}$, $-10^{3}$, $-10^{6}$, $-10^{9}$, $-10^{12}$},
    xminorticks = false,
    yminorticks = false,
    xlabel = {\# samples / evaluations},
    ylabel = {episodal reward},
    ylabel style   = {yshift = -.3em},
    scaled x ticks = false,
    x tick label style = {/pgf/number format/1000 sep={\,}},
    y tick label style = {/pgf/number format/1000 sep={\,}},
    cycle list name    = trainlist
  ]
  
    \addplot+[const plot] table[x index = 0, y expr = -\thisrowno{1}] {\tableREW};
    \addplot+[const plot] table[x index = 0, y expr = -\thisrowno{2}] {\tableREW};
    \addplot+[const plot] table[x index = 0, y expr = -\thisrowno{3}] {\tableREW};
  \end{loglogaxis}
\end{tikzpicture}%
  \tikzexternaldisable%

    \caption{Reward comparison.}
    \label{fig:cylinderwake_train_b}
  \end{subfigure}
  
  \vspace{.75\baselineskip}
  \tikzexternalenable%
  \tikzsetnextfilename{cylinderwake_train_legend}%
  \begin{tikzpicture}[font = \plotfontsize]
  \begin{axis}[%
    hide axis,
    width  = 1cm,
    height = 1cm,
    scale only axis,
    xmin = 0,
    xmax = 10,
    ymin = 0.5,
    ymax = 1.5,
    legend columns    = 2,
    legend cell align = {left},
    legend style      = {
      at     = {(0,0)},
      anchor = center,
      /tikz/every even column/.append style = {column sep = 0.5cm}}
  ]
    
    \pgfplotsset{cycle list name = trainlist, cycle list shift = 2}
    \pgfplotsinvokeforeach{1, 2, 3}{\addplot coordinates {(0,0)};}
    \addlegendentry{reinforcement learning}
    \addlegendentry{context-aware control, excl. eigenbasis estimation}
    \addlegendentry{\makebox[0pt][l]{context-aware control, incl.
      eigenbasis estimation}}
  \end{axis}
\end{tikzpicture}%
  \tikzexternaldisable%

  \vspace{0\baselineskip}
  
  \caption{Flow behind a cylinder:
    We restrict the reinforcement learning to only $20$ training episodes
    because obtaining samples from the linearized model is computationally
    expensive.
    The proposed approach requires around $38\times$ fewer data samples than
    needed in the very first DDPG episode to estimate a basis of the unstable
    dynamics and to infer a stabilizing controller in this example.}
  \label{fig:cylinderwake_train}
\end{figure}

We consider a laminar flow behind an obstacle.
The flow changes from its smooth steady state to vortex shedding under
arbitrarily small disturbances.
Snapshots of the flow velocity at end time are shown in
\Cref{fig:cylinderwake_dt}.
The flow dynamics are described by the Navier-Stokes equations.
The model we consider is a system of differential-algebraic equations with
$\nh = 22\,060$ states and $\np = 6$ control inputs, which allow to change the
flow velocities in a small region right behind the obstacle.
The matrices of the continuous-time model can be obtained from~\cite{BehBH17}.
Due to the presence of algebraic constraints in the model, we only consider the
discrete-time case, for which we obtain data with
implicit Euler for the linear terms and explicit Euler for the
quadratic terms with sampling time $\tau = 0.0025$.
The model has $\neigr = 2$ excitable unstable eigenvalues.

We need $655$ observations of the gradient to learn a basis of the
unstable dynamics.
Once a basis is obtained, we need $\dT = \neigr + 2 = 4$ state observations to
learn a stabilizing controller with the proposed approach.
\Cref{fig:cylinderwake_dt_b,fig:cylinderwake_dt_c} show the simulations
using controllers based on data from the linearized system and the
nonlinear system, respectively.
In both cases, the flow is smoothly stabilized towards the steady state.

For the DDPG agent, we stopped its training after only $20$ episodes due to high
computational costs resulting from the expensive evaluation of the linear model.
However, $20$ episodes are enough to demonstrate the superiority of the
proposed method in terms of data requirements as shown in
\Cref{fig:cylinderwake_train}.
The final DDPG-controller is not stabilizing the system.


\section{Conclusions}%
\label{sec:conclusions}

Building on the concept of context-aware learning~\cite{Peh19, WerP22,AlsP20}, we have
shown that for the task of finding a stabilizing controller, it is sufficient to
learn only the unstable dynamics of systems, in contrast to learning models of
the complete---stable and unstable---dynamics.
By combining the task of stabilization with learning system dynamics in a
context-aware fashion, we showed that there exist $\neigr$ states that are
sufficient to be observed for finding a guaranteed stabilizing controller, where
$\neigr$ is the dimension of the space that spans the unstable dynamics.
The dimension $\neigr$ is typically orders of magnitude lower than the dimension
of the state space of all dynamics.
These findings are leveraged by the proposed computational procedure that
estimates bases of spaces of unstable dynamics via samples from gradients
and that constructs stabilizing controllers from up to
$2\,000\times$ fewer state observations than traditional data-driven
control methods and variants of reinforcement learning.
This enables data-driven stabilization in applications with scarce data, such as
near rare events and when data generation is expensive.
In contrast to the reinforcement learning method used in the numerical
comparisons in this work, the proposed context-aware learning method certifies
the stabilization of the underlying system by the developed theory such that the
proposed context-aware controller inference provides trust for making
high-consequence decisions.
There are several avenues for future research.
One of them is combining context-aware learning with methods for constructing
manifolds of unstable dynamics~\cite{KraODetal05, ZieDG19}.
Working with manifolds, instead of spaces, can potentially be beneficial when
systems are strongly nonlinear.
At the same time, the nonlinear approximations introduced by manifolds can be
computationally challenging, especially for systems with high-dimensional state
spaces.


\section*{Acknowledgments}%
\addcontentsline{toc}{section}{Acknowledgments}

The authors acknowledge support from the Air Force Office of Scientific
Research (AFOSR) award FA9550-21-1-0222 (Dr. Fariba Fahroo).
The second author additionally acknowledges support from the National Science
Foundation under Grant No.~2012250.


\addcontentsline{toc}{section}{References}
\bibliographystyle{plainurl}
\bibliography{bibtex/myref}

\end{document}